\documentclass[11pt,a4paper]{article}

\RequirePackage{amsmath}
\RequirePackage{amssymb,latexsym}
\RequirePackage{amsthm}
\RequirePackage[hang,flushmargin,bottom,stable]{footmisc}
\RequirePackage[margin=0.35in,format=hang,font=small,labelfont=bf]{caption}
\RequirePackage{needspace}
\RequirePackage{graphicx,tikz}
\RequirePackage{changebar}
\RequirePackage[T1]{fontenc}
\RequirePackage[sc]{mathpazo}
\RequirePackage[colorlinks=true,urlcolor=blue,linkcolor=blue,citecolor=blue]{hyperref}

\hypersetup{
pdfstartview={XYZ null null 1.00},
pdftitle={On the evolution of random integer compositions},
pdfauthor={David Bevan and Dan Threlfall},
pdfdisplaydoctitle=true,
pdflang={en-GB}
}

\newtheorem{thm}{Theorem}[section]
\newtheorem{obs}[thm]{Observation}
\newtheorem{prop}[thm]{Proposition}
\newtheorem{question}[thm]{Question}

\newcommand{\bbN}{\mathbb{N}}
\newcommand{\C}{\mathbf{C}}
\newcommand{\CCC}{\mathcal{C}}
\newcommand{\ccnm}[1][m]{\CCC_{n,#1}}
\newcommand{\cnm}[1][m]{\C_{n,#1}}
\newcommand{\cnp}[1][p]{\C_{n,#1}}
\newcommand{\ct}[1][t]{\C_{#1}}
\newcommand{\defeq}{\mathchoice{\;:=\;}{:=}{:=}{:=}}
\newcommand{\eq}{\mathchoice{\;=\;}{=}{=}{=}}
\newcommand{\expec}[1]{\mathbb{E}\big[#1\big]}
\newcommand{\floor}[1]{\left\lfloor #1 \right\rfloor}
\newcommand{\geqs}{\geqslant}
\newcommand{\HIDE}[1]{}
\newcommand{\leqs}{\leqslant}
\newcommand{\liminfinfty}[1][n]{\liminf\limits_{#1\rightarrow\infty}}
\newcommand{\liminfty}[1][n]{\lim\limits_{#1\rightarrow\infty}}
\newcommand{\limsupinfty}[1][n]{\limsup\limits_{#1\rightarrow\infty}}
\newcommand{\prob}[1]{\mathbb{P}\big[#1\big]}
\newcommand{\QQQ}{\mathcal{Q}}
\newcommand{\rulebreak}{\vspace{24pt}\hrule\vspace{15pt}}
\newcommand{\sgeqs}{\mathchoice{\;\geqs\;}{\geqs}{\geqs}{\geqs}}
\newcommand{\sgg}{\mathchoice{\;\gg\;}{\gg}{\gg}{\gg}}
\newcommand{\sgsim}{\mathchoice{\;\gtrsim\;}{\gtrsim}{\gtrsim}{\gtrsim}}
\newcommand{\sgtr}{\mathchoice{\;>\;}{>}{>}{>}}
\newcommand{\sleqs}{\mathchoice{\;\leqs\;}{\leqs}{\leqs}{\leqs}}
\newcommand{\sless}{\mathchoice{\;<\;}{<}{<}{<}}
\newcommand{\sll}{\mathchoice{\;\ll\;}{\ll}{\ll}{\ll}}
\newcommand{\slsim}{\mathchoice{\;\lesssim\;}{\lesssim}{\lesssim}{\lesssim}}
\newcommand{\ssim}{\mathchoice{\;\sim\;}{\sim}{\sim}{\sim}}
\newcommand{\veps}{\varepsilon}

\newcommand{\tmax}{\mathsf{max}}    
\newcommand{\tmin}{\mathsf{min}}    

\newcommand{\cmax}{\mathsf{comp}_{\max}}    
\newcommand{\cmin}{\mathsf{comp}_{\min}}    
\newcommand{\gmax}{\mathsf{gap}_{\max}}     
\newcommand{\gmin}{\mathsf{gap}_{\min}}     

\newcommand{\ep}[1][\pi]{\,{}^{\,{=}}#1}                
\newcommand{\gp}[1][\pi]{\,{}^{\geqs}#1}                
\newcommand{\lp}[1][\pi]{\,{}^{\leqs}#1}                
\newcommand{\epc}[1][\pi]{\,{}^{{=}}\overline{#1}}      
\newcommand{\gpc}[1][\pi]{\,{}^{\geqs}\overline{#1}}    
\newcommand{\lpc}[1][\pi]{\,{}^{\leqs}\overline{#1}}    
\newcommand{\opc}[1][\pi]{\overline{#1}}                

\newcommand{\plotc}[2]  
{
  \draw[thick] (0,0)--(#1,0);
  \foreach \y [count=\x] in {#2}
  {
    \draw (\x-.5,-1.15) node{\footnotesize\y};
    \ifnum0=\y {} \else {
      \filldraw[thick,fill=gray!25] (\x-1,0) rectangle (\x,\y);
    } \fi
  }
}

\title{\textbf{On the evolution of random integer compositions}}
\author{David Bevan${}^\dagger$ and Dan Threlfall${}^{\dagger\ddag}$}
\date{}

\begin{document}
\maketitle

{\begin{NoHyper}
\let\thefootnote\relax\footnotetext
{${}^\dagger$Department of Mathematics and Statistics, University of Strathclyde, Glasgow, Scotland.}
\let\thefootnote\relax\footnotetext
{${}^\ddag$The second author was supported by an EPSRC Mathematical Sciences Studentship.}
\end{NoHyper}}

{\begin{NoHyper}
\let\thefootnote\relax
\end{NoHyper}}

\begin{abstract}
\noindent
We explore how the asymptotic structure of a random $n$-term weak integer composition of $m$ evolves, as $m$ increases from zero.
The primary focus is on establishing thresholds for the appearance and disappearance of 
substructures.
These include the longest and shortest
runs of zero terms or of nonzero terms,
longest increasing runs,
longest runs of equal terms,
largest squares (runs of $k$ terms each equal to~$k$),
as well as a wide variety of other patterns. 
Of particular note is the dichotomy between the appearance and disappearance of exact consecutive patterns,
with {\emph{smaller}} patterns appearing before {\emph{larger}} ones,
whereas
{\emph{longer}} patterns disappear before {\emph{shorter}} ones.
\end{abstract}

\section{Introduction}\label{sectIntro}

We initiate 
the study of integer compositions from an evolutionary perspective, in an analogous manner to the Gilbert--Erd\H{o}s--R\'enyi random graph~\cite{ER1959,ER1960,Gilbert1959}.
Our two primary models are the \emph{uniform random composition} $\cnm$, drawn uniformly 
from the family of $n$-term weak integer compositions of $m$, and 
the \emph{geometric random composition} $\cnp$, an $n$-term weak integer composition in which each term is sampled independently from the geometric distribution with parameter $q=1-p$; that is, $\prob{\cnp(i)=k}=qp^k$ for each $k\geqs0$ and $i\in[n]$.
We are interested in how, for large $n$, the structure of $\cnm$ or $\cnp$ evolves as $m$ or $p$ increases from zero.
The primary emphasis is on the establishment of thresholds for the appearance and disappearance of small substructures.

\begin{table}[p]
\begin{center}
\small
  \renewcommand*{\arraystretch}{1.2}
  \begin{tabular}{|l|l|l|}
    \hline
    $m$ && {\small Prop.} \\ \hline
    \hline    
    $n^{1/4}$               & shortest gap has length $\leqs\sqrt{n}$                       & \ref{propShortestComp2}           \\ \hline
    $n^{1/2}/\sqrt{\log n}$ & shortest gap has length $\leqs\log n$                         & \ref{propShortestComp2}           \\ \hline
    $n^{1/2}$               & longest component has length 2                                & \ref{propLongestComp1}            \\
                            & shortest gap has length 1                                     & \ref{propShortestComp1}           \\ \hline
    $n^{1/2}\log n$         & longest gap has length $\leqs\sqrt{n}$                        & \ref{propLongestComp4}            \\ \hline
    $n^{2/3}$              
                            & largest term equals 3                                         & \ref{propLargestTerm}             \\
                            & increasing run of length 3                                    & \ref{propTotalOrderingPattConsec} \\
                            & exact nonconsecutive patterns with largest term 3             & \ref{propExactPatt}               \\
                            & increasing subsequence of length 4                            & \ref{propTotalOrderingPatt}       \\ \hline
    $n^{5/6}$               & exact consecutive patterns of size 6                          & \ref{propExactPattConsec}         \\
                            & increasing run of length 4                                    & \ref{propTotalOrderingPattConsec} \\ \hline
    $\gg n^{1-\veps}$       & any exact consecutive pattern                                 & \ref{propExactPattConsec}         \\ \hline
    $n/\log n$              & longest component has length $\geqs\log n/\log\log n$         & \ref{propLongestComp2}            \\ \hline
    $n$                     & longest component and longest gap have length $\log n$        & \ref{propLongestComp3}            \\
                            & largest term $\geqs\log n$                                    & \ref{propLargestTermConstantP}    \\
                            & each total ordering of consecutive terms equally likely       & \ref{propTotalOrderProbLimit}     \\ \hline
    $n\log n/\log\log n$    & largest possible square pattern, length $\geqs\log n/\log\log n$   & \ref{propSquarePat}               \\ \hline
    $n\log n$               & every gap has length $\leqs\log n/\log\log n$                 & \ref{propLongestComp2}            \\ \hline
    $n^{4/3}$               & longest gap has length 2                                      & \ref{propLongestComp1}            \\
                            & no given exact consecutive pattern of length 4                & \ref{propExactPattConsec}         \\ \hline
    $n^{3/2}/\log n$        & longest component has length $\geqs\sqrt{n}$                  & \ref{propLongestComp4}            \\
                            & largest term $\geqs \sqrt{n}$                                 & \ref{propLargestTermSmallQ}       \\ \hline
    $n^{3/2}$               & every gap has length 1                                        & \ref{propLongestComp1}            \\
                            & length of shortest component exceeds any fixed value          & \ref{propShortestComp1}           \\
                            & no run of 3 equal terms                                       & \ref{propCarlitzThreshold}        \\ \hline
    $n^{3/2} \sqrt{\log n}$ & every component has length $\geqs\log n$                      & \ref{propShortestComp2}           \\ \hline
    $n^{7/4}$               & every component has length $\geqs\sqrt{n}$                    & \ref{propShortestComp2}           \\ \hline
    $n^2/\log n$            & largest term $\geqs n$                                        & \ref{propLargestTermSmallQ}       \\ \hline
    $n^2$                   & no gaps (single component)                                    & \ref{propLongestComp1}            \\
                            & smallest term exceeds any fixed value                         & \ref{propSmallestTerm}            \\
                            & no balanced peaks or valleys                                  & \ref{propOrderPattThreshold}      \\
                            & no adjacent equal terms (Carlitz composition)                 & \ref{propCarlitzThreshold}        \\
                            & no given nonconsecutive exact pattern                         & \ref{propExactPatt}               \\ \hline
    $n^{5/2}$               & no three terms equal                                          & \ref{propEqualTerms}              \\ \hline
    $n^3$                   & every term $\geqs n$                                          & \ref{propSmallestTermGrowing}     \\
                            & every term distinct                                           & \ref{propEqualTerms}              \\ \hline
  \end{tabular}
\end{center}
\caption{Some thresholds encountered during the evolution of~$\cnm$}\label{tableResults}
\end{table}


A selection of our results is presented in evolutionary order in Table~\ref{tableResults} on page~\pageref{tableResults} .
Here we list a few highlights.
A property 
of compositions holds \emph{with high probability} (\emph{w.h.p.}) 
if
asymptotically the probability of it holding tends to one. 
We defer the formal presentation of this and other definitions and notational conventions until later.

\vspace{-9pt}
\begin{enumerate}\itemsep0pt

\item
A \emph{gap} is a maximal run of zero terms.
As long as $m\ll n^{1/2}$, with high probability
the length of the \emph{shortest} {gap} exceeds any fixed value.
However, as soon as $m\gg n^{1/2}$, w.h.p.
there are gaps of length 1.
(Section~\ref{sectShortestComponent})

\item
In contrast, w.h.p.
the length of the \emph{longest} gap exceeds any fixed value as long as \mbox{$m\ll n^{1+\delta}$} for all $\delta>0$.
But once $m\gg n^{1+1/k}$, w.h.p.
$\cnm$ has no gap of length $k$ or greater.
In particular, once $m\gg n^{3/2}$, w.h.p.
every gap has length 1, and when $m\gg n^2$, w.h.p.
$\cnm$ no longer has any gaps.
(Section~\ref{sectLongestComponent})

\item
There is a dichotomy between the thresholds for the arrival and for the departure of \emph{exact consecutive patterns} as $\cnm$ evolves,
the former being ordered by \emph{size}, \emph{smaller} patterns appearing before \emph{larger} ones,
and the latter being ordered by \emph{length}, \emph{longer} patterns disappearing before \emph{shorter} ones.
(Proposition~\ref{propExactPattConsec})

\item
If $n^{1-\delta}\ll m\ll n^{1+\delta}$ for every $\delta>0$, then w.h.p. {any} given exact consecutive pattern occurs in $\cnm$.
In contrast, once $m\gg n^2$, then w.h.p. any specific exact consecutive pattern is absent from $\cnm$.
(Proposition~\ref{propExactPattConsec})

\item
The {largest} \emph{square} (run of $k$ terms each equal to $k$) that we expect to see in the evolution of $\cnm$
has side length a little greater than ${\log n}/{\log\log n}$, and is seen when \mbox{$m\sim {n\log n}/{\log\log n}$}.
(Proposition~\ref{propSquarePat})

\item
As long as $m\ll n$, w.h.p. the \emph{largest term} in $\cnm$ exhibits \emph{two-point concentration} (that is, it takes one of only two possible values).
(Proposition~\ref{propLargestTermSmallP})

\item
As long as $m\ll n^2$, w.h.p. the \emph{smallest term} in $\cnm$ equals zero, but $m\sim n^2$ is the threshold for the smallest term to exceed \emph{any} fixed positive value.
(Proposition~\ref{propSmallestTerm})

\item
Once $m\gg n$, the relative ordering of any $k$ consecutive terms of $\cnm$ is asymptotically uniformly distributed over the $k!$ permutations of length~$k$.
(Proposition~\ref{propTotalOrderProbLimit})

\item
A pattern specifying the relative ordering of $k$ consecutive terms which has a repeated term and has largest term equal to $r$
exhibits
a threshold
for its disappearance from $\cnm$
at \mbox{$m\sim n^{1+1/d}$}, where $d=k-1-r$.
(Proposition~\ref{propOrderPattThreshold})

\item
Once $m\gg n^2$, w.h.p. $\cnm$ has no adjacent equal terms (that is, it is a \emph{Carlitz} composition).
(Proposition~\ref{propCarlitzThreshold})

\item
The threshold for the appearance of a \emph{nonconsecutive exact pattern} depends only on its maximum term.
In contrast, all such patterns share the same threshold $m\sim n^2$ for their disappearance.
(Proposition~\ref{propExactPatt})

\item
The threshold for the appearance of a \emph{vincular} pattern depends on the \emph{size} of its \emph{largest} block.
In contrast, the threshold for the \emph{disappearance} of such a pattern depends on the \emph{length} of its \emph{longest} block.
(Proposition~\ref{propVincularPatt})

\item
Once $m\gg n^3$, w.h.p. every term of $\cnm$ is distinct.
(Proposition~\ref{propEqualTerms})

\end{enumerate}
\vspace{-9pt}

The random composition doesn't exhibit a spectacular phase transition like that seen in the random graph with the appearance of the giant component.
Perhaps the most dramatic change occurs when $m\sim n^2$ with the disappearance of both the last gap and also of the last pair of adjacent equal terms
as the smallest term jumps from zero to exceed any fixed value.

There are many questions that we don't address in this introductory paper.
One example would be to determine thresholds for the presence of components or gaps of equal length.
Also, we only determine threshold probabilities for certain properties.
It would certainly be interesting to establish them for arbitrary nonconsecutive ordering patterns, which we don't investigate in detail.
We also raise a general question concerning our models:
Does every nontrivial monotone property of compositions have a threshold? (Question~\ref{questMonotone})

In another direction, in a subsequent paper we build on this work to determine thresholds for the appearance of patterns in the evolution of random \emph{permutations}~\cite{BTPermutations}, in a similar manner to the work of Acan and Pittel~\cite{AP2013}.
A further possibility would be to extend our models with additional structure in order to represent {random} \emph{multigraphs}, each term in a composition recording the number of edges between a pair of vertices (see~\cite[Theorems 2.9 and 2.10]{CZL2012}).
For example, the threshold for a random multigraph to cease being simple and the threshold for the presence of a complete spanning subgraph follow directly from results in this paper.


The evolution of the random \emph{graph} is now a classical topic, with three graduate textbooks available~\cite{Bollobas2001,FK2015,JLR2000}.
In contrast,
rather surprisingly,
other combinatorial objects do not appear to have been investigated from this perspective, with the notable exception of the work of Acan and Pittel~\cite{AP2013} on the threshold for connectivity in the evolution of a random permutation.
Existing work on the structure of large compositions either has an enumerative flavour (determining generating functions), or else it takes a probabilistic perspective, either considering the uniform distribution over all (non-weak) compositions of $n$ or investigating sequences of geometrically distributed random variables. 
For the most part our approach is somewhat orthogonal to both of these.
Specific intersecting works are cited below in the relevant sections.

In 
the next section,
we define our random models and explore their relationships,
and introduce the notions of a property and a threshold 
in these models.
We also present three key tools: the First Moment Method and the Second Moment Method which we use to determine the location of thresholds, and the Chen--Stein Method which gives us the probability of a property holding at its threshold.
In Section~\ref{sectComponents}, our focus is on \emph{components} (maximal runs of nonzero terms) and \emph{gaps} (maximal runs of zero terms), establishing thresholds for the length of the longest 
or shortest component to exceed some value, with analogous results for the length of gaps.
Finally, Section~\ref{sectPatterns} concerns the appearance and disappearance of a wide variety of types of pattern, including \emph{exact} patterns (in which terms must take specified
values), \emph{upper} and \emph{lower} patterns (in which terms are bounded below or above), and \emph{ordering} patterns (which specify the relative ordering of terms).

If $f$ and $g$ are positive functions of $n$, then
we use the following 
notation:
\begin{align*}
  f\lesssim g                   & \text{~~~~if~~} \limsupinfty f/g<\infty , \\
  f \asymp g                    & \text{~~~~if~~} 0<\liminfinfty f/g \text{~~and~~} \limsupinfty f/g<\infty , \\
  f \sim g                      & \text{~~~~if~~} \liminfty f/g=1 , \\
  f \sim 0                      & \text{~~~~if~~} \liminfty f=0 , \\
  f\ll g \text{~~or~~} g\gg f   & \text{~~~~if~~} \liminfty f/g=0 .
\end{align*}
In particular, $f\ll1$ if $\liminfty f=0$, and $f\gg1$ if $\liminfty f=\infty$.

Note that we avoid the use of ``big O'' notation, preferring something more intuitive. We have the following equivalences:
\[
f\lesssim g \:\Longleftrightarrow\: f=O(g),
\qquad
f \asymp g \:\Longleftrightarrow\: f=\Theta(g),
\qquad
f \ll g \:\Longleftrightarrow\: f=o(g).
\]
Note also that $f \sim 0$ is nonstandard; however, it significantly simplifies the presentation of our results.

\section{Random compositions}\label{sectModels}

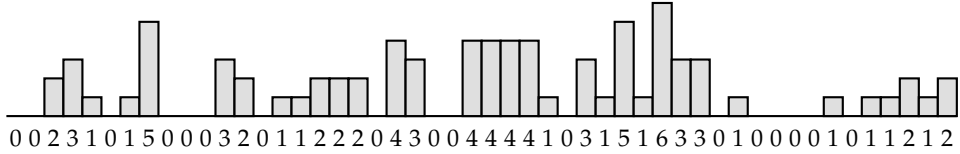
\begin{figure}[t]
\begin{center}
\begin{tikzpicture}[scale=0.25]
  \plotc{50}{0,0,2,3,1,0,1,5,0,0,0,3,2,0,1,1,2,2,2,0,4,3,0,0,4,4,4,4,1,0,3,1,5,1,6,3,3,0,1,0,0,0,0,1,0,1,1,2,1,2}
\end{tikzpicture}
\end{center}
\vspace{-6pt}
\caption{Bar-chart representation of a $50$-term composition of 80}\label{figComposition}
\end{figure}

An \emph{$n$-term weak composition of $m$}, or just an \emph{$n$-composition of~$m$}, is a sequence of $n$ nonnegative integers $(c_1,\ldots,c_n)$ such that $\sum_{i=1}^nc_i=m$.
Compositions can be considered to be words over the nonnegative integers, and, if no term exceeds nine, we sometimes write specific compositions simply as a sequence of digits.
See Figure~\ref{figComposition} for an example.
Alternatively, we can consider such a composition to consist of a sequence of $n$ \emph{boxes}, such that for each $i\in[n]\defeq\{1,2,\ldots,n\}$, box $i$ contains $c_i$ \emph{balls}. 
By a simple ``stars and bars'' argument, it can be seen that the number of distinct $n$-compositions of $m$ is $\binom{m+n-1}{m}$.

If $C$ is an integer composition, then we use $C(i)$ to denote its $i$th term, and $|C|$ to denote its \emph{size}, the sum of its terms.
Let $\CCC_n$ denote the set of all $n$-compositions, and
let $\ccnm$ be the set of all $n$-compositions of $m$.
We now present three models of random integer compositions.

\subsection{The uniform random composition \texorpdfstring{$\cnm$}{Cnm}}\label{sectCnm}

The \emph{uniform random composition} $\cnm$ is drawn uniformly from $\ccnm$.
Thus, for every composition $C\in\ccnm$,
\[
\prob{\cnm=C} \eq  \binom{m+n-1}{m}^{\!-1} ,
\]
each of the 
distinct $n$-compositions of $m$ being equally likely.
In statistical quantum mechanics, this is known as \emph{Bose--Einstein statistics}, modelling the distribution followed by \emph{bosons}. See~\cite[Example 12.2]{Janson2012} for an exposition in the context of a variety of different balls-in-boxes models, and~\cite{Huillet2011} for work more closely related to our concerns.

\subsection{The evolutionary random composition \texorpdfstring{$\ct$}{Ct}}\label{sectCt}

An alternative, {evolutionary}, perspective comes from taking a dynamic view and considering a process on compositions,
namely an infinite sequence of compositions,
$
0^n, C_1, C_2, C_3, \ldots,
$
where $0^n$ denotes the \emph{empty} $n$-composition $(0,\ldots,0)$, and $C_{t+1}$ is obtained from $C_t$ by the addition of 1 to a single term.
Note that there is no maximal $n$-composition (unlike the situation with random graphs).

The \emph{evolutionary random composition} $(\ct)_{t\geqs0}$ is the Markov chain satisfying $\C_0=0^n$ and, for each $t\geqs0$ and $j\in[n]$,
\[
\prob{\C_{t+1}=\ct^{+j}} \eq  \frac{\ct(j)+1}{n+t} ,
\]
where $C^{+j}$ denotes the composition obtained from $C$ by the addition of 1 to its $j$th term.
The evolutionary random composition $\ct$ is uniformly distributed over $n$-compositions of $t$:
\begin{prop}\label{propCtUniformity}
For each $t\geqs0$, the random composition $\C_t$ is uniformly distributed over $\ccnm[t]$.
\end{prop}
\begin{proof}
  We use induction on $t$. Trivially, $\C_0$ is uniformly distributed over $\ccnm[0]$.
    Suppose $\C_t$ is uniformly distributed over $\ccnm[t]$,
    and that $C\in\ccnm[t+1]$. 
    Let $C^{-j}$ denote the composition obtained from $C$ by the subtraction of 1 from its $j$th term (if this is possible).
    Then,
    \begin{align*}
      \prob{\C_{t+1}=C}
      &\eq \sum_{j\in[n],\, C(j)\neq0} \prob{\C_t=C^{-j}} \frac{C(j)}{n+t} \\
      &\eq \binom{n+t-1}{t}^{\!-1} \sum_{j\in[n]} \frac{C(j)}{n+t} 
       \eq \frac{t!\,(n-1)!}{(n+t-1)!}\,\frac{t+1}{n+t}
       \eq \binom{n+t}{t+1}^{\!-1} . \qedhere
    \end{align*}
\end{proof}

The evolutionary random composition corresponds to the following multicoloured P\'olya urn model:
Consider an urn initially containing one ball of each of $n$ different colours.
Balls are drawn at random one at a time, and after drawing the ball, it is replaced together with another of the same colour.
Let $Y_i$ be the number of balls of colour $i$ in the urn after $t$ draws. Then, $(Y_1-1,Y_2-1,\ldots,Y_n-1)$ has the same distribution as $\C_t$.
See~\cite{Holst1979} and~\cite[Example~12.4]{Janson2012}.

\subsection{The geometric random composition \texorpdfstring{$\cnp$}{Cnp}}\label{sectCnp}

If $p\in[0,1)$, then the \emph{geometric random composition} $\cnp$ is distributed over $\CCC_n$ so that for each $C\in\CCC_n$, we have $\prob{\cnp=C}=q^n p^{|C|}$, where
$q=1-p$.
Each term of $\cnp$ is sampled independently from the geometric distribution with parameter $q$; that is, $\prob{C(i)=k}=qp^k$ for each $k\geqs0$ and $i\in[n]$.
Note that $\cnp$ is not defined for $p=1$.

To avoid unnecessary repetition, when considering $\cnp$ in this work, $q$ always denotes $1-p$.
Moreover, we also assume that the definition of
any annotated $p$ also defines a similarly annotated $q$, so we have $q_1=1-p_1$ and $q^+=1-p^+$, without stating so explicitly.

We collect here a few basic facts about $\cnp$.
Each term has mean $p/q$ and variance $p/q^2$, and
its size $|\cnp|$ satisfies a negative binomial distribution,
\[
\prob{|\cnp|=m} \eq  \binom{m+n-1}{m} p^mq^n ,
\]
with mean $\mu_{n,p}=np/q$.
Note that if $p\ll 1$ then $\mu_{n,p}\sim np$, and if $q\ll 1$ then $\mu_{n,p}\sim n/q$.

The size $|\cnp|$ has variance $np/q^2$, and so exhibits a concentrated distribution as long as $p\gg n^{-1}$.
In particular, by Chebyshev's inequality, we have the following:
\begin{obs}\label{obsCnpConcentration}
  $\prob{\big| |\cnp|-np/q \big|\geqs\alpha\sqrt{np}/q} \leqs \alpha^{-2}$.
\end{obs}

  Geometric random compositions
  of size $m$
  are uniformly distributed over $\ccnm$: 
\begin{prop}\label{propCnpUniformity}
  A random composition $\cnp$
  whose terms sum to $m$
  is equally likely to be any one of the 
  distinct $n$-compositions of $m$.
\end{prop}
\begin{proof}
If $C_1,C_2\in\ccnm$, then
$
\prob{\cnp=C_1} \eq p^mq^n \eq \prob{\cnp=C_2}. 
$
\qedhere
\end{proof}
Thus, $\cnp$ conditioned on the event $|\cnp|=m$ is equal in distribution to $\cnm$.
Note that this holds for {any} choice of $p$ and $m$.
As is the case with random graphs, $\cnp$ is more amenable to analysis than $\cnm$, so we prefer to work with $\cnp$ and then transfer the results to~$\cnm$ (see Propositions~\ref{propMonotoneAsymptEquiv}, \ref{propThresholdTransfer} and~\ref{propExactPattCnmEqCnpProb} below).

\subsection{Properties}\label{sectProperties}

We consider a \emph{property} 
of $n$-compositions simply to be a subset of $\CCC_n$. 
For example,
the set of $n$-compositions with no zero terms
is a property, as is
the set of $n$-compositions with at least one term equal to three.

A property $\QQQ$ is
\emph{increasing}
if $C\in\QQQ$ implies $C^{+j}\in\QQQ$ for every $j\in[n]$, or equivalently
if $C\in\QQQ$ implies $C+C'\in\QQQ$ for any $C'\in\CCC_n$, where $C+C'$ denotes the term-wise sum of two $n$-compositions.
The complement of an increasing property is \emph{decreasing}.
A property that is either increasing or decreasing is \emph{monotone}.
For example, the $n$-compositions with no zero terms
form an increasing property, whereas the set of $n$-compositions with at least one term equal to three is not monotone.
Both $\cnm$ and $\cnp$ behave monotonically with respect to monotone properties:

\begin{prop}\label{propCnmIsMonotone}
  If $\QQQ$ is an increasing property and $m_1< m_2$, then $\prob{\cnm[m_1]\in\QQQ}\leqs\prob{\cnm[m_2]\in\QQQ}$.
\end{prop}
\begin{proof}
    $\prob{\C_{t+1}\in\QQQ} \geqs \prob{\C_{t+1}\in\QQQ
    \:\wedge\:
    \C_t\in\QQQ} = \prob{\C_t\in\QQQ}$, since $\QQQ$ is increasing.
\end{proof}

\begin{prop}\label{propCnpIsMonotone}
  If $\QQQ$ is an increasing property and $p_1< p_2$, then $\prob{\cnp[p_1]\in\QQQ}\leqs\prob{\cnp[p_2]\in\QQQ}$.
\end{prop}
\begin{proof}
  Let
  $\cnp[p_1,p_2]$ denote a random $n$-composition, each of whose terms is sampled independently from the following distribution. For each $i\in[n]$,
  \[
    \prob{\cnp[p_1,p_2](i)=k} \eq
    \begin{cases}
      \frac{q_2}{q_1}, & \text{if $k=0$} , \\[6pt]
      \frac{q_2}{q_1}\big(1-\frac{p_1}{p_2}\big)p_2^{\,\,k}, & \text{if $k\geqs1$} .
    \end{cases}
  \]
  We claim that $\cnp[p_1]+\cnp[p_1,p_2]$ has the same distribution as $\cnp[p_2]$ if 
  $\cnp[p_1]$ and $\cnp[p_1,p_2]$ are chosen independently, thus providing a way of building $\cnp[p_2]$ from $0^n$ in two steps via $\cnp[p_1]$.

  To prove this equality of distribution, we use probability generating functions. Let
  \begin{align*}
    f_p(x) & \eq \sum_{k\geqs0}\prob{\cnp(i)=k}x^k \eq  \frac{q}{1-px} , \\
    f_{p_1,p_2}(x) & \eq \sum_{k\geqs0}\prob{\cnp[p_1,p_2](i)=k}x^k \eq  \frac{q_2(1-p_1x)}{q_1(1-p_2x)} .
  \end{align*}
  Thus $f_{p_2}(x)=f_{p_1}(x)f_{p_1,p_2}(x)$, and equality of distribution then follows from the independence of each term in the random compositions.

  Hence, by coupling $\cnp[p_1]$ and $\cnp[p_2]$,
  \begin{align*}
    \prob{\cnp[p_2]\in\QQQ} &\eq \prob{\cnp[p_1]+\cnp[p_1,p_2]\in\QQQ} \\
                              &\sgeqs \prob{\cnp[p_1]+\cnp[p_1,p_2]\in\QQQ
                              \:\wedge\:
                              \cnp[p_1]\in\QQQ} & \text{(where the $\cnp[p_1]$ are the same)}\\
                              &\eq \prob{\cnp[p_1]\in\QQQ} ,
  \end{align*}
  since $\QQQ$ is increasing.
\end{proof}
A property $\QQQ$ is \emph{convex} if $C\in\QQQ$ and $C+C_1+C_2\in\QQQ$ implies $C+C_1\in\QQQ$.
For example, the property of having exactly one zero term is convex.
Every convex property is the intersection of an increasing property and a decreasing property.

Typically, we are interested in whether a property holds, or fails to hold, in the asymptotic limit.
We say that $\QQQ$ holds \emph{asymptotically almost surely} (\emph{a.a.s.}) or, synonymously, \emph{with high probability} (\emph{w.h.p.}) in $\cnp$ if 
$\prob{\cnp\in\QQQ}\sim1$,
and analogously for $\cnm$.\label{defWHP} If a property holds a.a.s., then its complement asymptotically \emph{almost never} holds.

Since $|\cnp|$ is concentrated around its mean, it is reasonable to expect that, if $n$ is large, then $\cnp$ and $\cnm$ should behave in a similar fashion when $m\sim np/q$, or equivalently, when $p\sim m/(m+n)$.
This is indeed the case, and the following proposition enables us to transfer results from $\cnp$ to $\cnm$, the probability that an increasing property holds being the same in both models.
\begin{prop}\label{propMonotoneAsymptEquiv}
  Let $\QQQ$ be an increasing property and $\alpha\in[0,1]$ be a constant.
  Suppose $p_0=p_0(n)$ and
  $\delta=\delta(n)\gg{\sqrt{p_0}}/{(q_0\sqrt{n})}$
  are such that
  $\prob{\cnp\in Q}\sim \alpha$
  for all $p$
  for which $p/q$ differs from $p_0/q_0$ by no more than $\delta$.
  Then
  $\prob{\cnm[m_0]\in Q}\sim\alpha$,
  where $m_0=np_0/q_0$.
\end{prop}
\begin{proof}
  Let $p^-$ satisfy $p^-/q^-=p_0/q_0-\delta$, 
  and $p^+$ satisfy $p^+/q^+=p_0/q_0+\delta$. 

  Fix any $\veps>0$ and suppose $n$ is sufficiently large that both
  $\prob{|\cnp[p^-]|> m_0}\leqs \veps$
  and
  $\prob{|\cnp[p^+]|< m_0}\leqs \veps$.
  This is possible by Observation~\ref{obsCnpConcentration} given that
  \begin{align*}
  m_0-\frac{np^-}{q^-} &\eq n\delta \sgg \frac{\sqrt{np_0}}{q_0} \ssim \frac{\sqrt{np^-}}{q^-} , \text{~~~and} \\
  \frac{np^+}{q^+}-m_0 &\eq n\delta \sgg \frac{\sqrt{np_0}}{q_0} \ssim \frac{\sqrt{np^+}}{q^+} .
  \end{align*}
  Then,
  \begin{align*}
  \prob{\cnp[p^-]\in\QQQ}
  &\eq \sum_{k\leqslant m_0} \prob{\cnm[k]\in\QQQ} \, \prob{|\cnp[p^-]|=k} \:+\: \sum_{k>m_0} \prob{\cnm[k]\in\QQQ} \, \prob{|\cnp[p^-]|=k} \\
  &\sleqs \prob{\cnm[m_0]\in\QQQ} \, \prob{|\cnp[p^-]|\leqslant m_0} \:+\: \prob{|\cnp[p^-]|>m_0} \\
  &\sleqs \prob{\cnm[m_0]\in\QQQ} \:+\: \veps ,
  \end{align*}
  by Proposition~\ref{propCnmIsMonotone}, since $\QQQ$ is increasing. Thus, $\liminfty\prob{\cnm[m_0]\in Q}\geqs\alpha$.

  Similarly,
  \begin{align*}
      \prob{\cnp[p^+]\in\QQQ}
      &\sgeqs \sum_{k\geqslant m_0} \prob{\cnm[k]\in\QQQ} \, \prob{|\cnp[p^+]|=k} \\
      &\sgeqs \prob{\cnm[m_0]\in\QQQ} \, \prob{|\cnp[p^+]|\geqslant m_0} \\
      &\sgeqs (1-\veps)\,\prob{\cnm[m_0]\in\QQQ} ,
  \end{align*}
  and so $\liminfty\prob{\cnm[m_0]\in Q}\leqs\alpha$.
\end{proof}
Note that, in general, we can't remove the requirement that the property must be increasing.
For example, if $\QQQ$ is the set of $n$-compositions with no zero terms whose size is not a perfect square,
then $\prob{\cnm\in\QQQ}=0$ whenever $m$ is a perfect square, whereas $\QQQ$ holds a.a.s. in $\cnp$ once $p$ is sufficiently large that both conditions hold w.h.p.
However, in some situations we can transfer results concerning non-monotone properties from $\cnp$ to~$\cnm$.
For example, Proposition~\ref{propExactPattCnmEqCnpProb} enables us to do this for exact consecutive patterns.

\subsection{Thresholds}\label{sectThresholds}

One of the most striking observations regarding the evolution of large random combinatorial objects is the abrupt nature of the appearance and disappearance of many properties.
We say that a function $m^\star=m^\star(n)$ is a \emph{threshold} for an increasing
  property $\QQQ$ in 
  $\cnm$~if
  \begin{align*}
  \prob{\cnm\in\QQQ} & \ssim
  \begin{cases}
    0 & \text{if~ $m\ll m^\star$,} \\
    1 & \text{if~ $m\gg m^\star$,}
  \end{cases}
  \end{align*}
  and that
  $p^\star=p^\star(n)$ or $q^\star=q^\star(n)$ is a {threshold} for $\QQQ$ in 
  $\cnp$~if
  \begin{align*}
  \prob{\cnp\in\QQQ} & \ssim
  \begin{cases}
    0 & \text{if~ $p/q\ll p^\star/q^\star$,} \\
    1 & \text{if~ $p/q\gg p^\star/q^\star$.}
  \end{cases}
  \end{align*}
Note here that $p_1/q_1\ll p_2/q_2$ if and only if either $p_1\ll p_2$ (and $q_1\sim1$) or else $q_1\gg q_2$ (and $p_2\sim1$).
  
That is, a property asymptotically almost never holds below its threshold, but holds asymptotically almost surely above it.
\emph{A priori}, there is no reason why a property should have a threshold.
Nevertheless, it turns out that thresholds exist for
every nontrivial monotone property of the subsets of a set, and hence also of graphs (see~\cite{BT1987,FK1996}).
For a recent short exposition, see~\cite{Park2023}.
Note, however, that this result does not encompass our models.

\begin{question}\label{questMonotone}
  Does every nontrivial monotone property of compositions have a threshold?
\end{question}

In many situations, it can be determined that the threshold is more abrupt.
A function $m^\star$ is a \emph{sharp} threshold for a 
  property $\QQQ$ in 
  $\cnm$,
  and $p^\star$ is a {sharp} threshold for $\QQQ$ in 
  $\cnp$,
  if, for every $\veps>0$,
 \begin{align*}
  \prob{\cnm\in\QQQ} & \ssim
  \begin{cases}
    0 & \text{if~ $m\leqs(1-\veps) m^\star$,} \\
    1 & \text{if~ $m\geqs(1+\veps) m^\star$,}
  \end{cases} \\[3pt]
  \prob{\cnp\in\QQQ} & \ssim
  \begin{cases}
    0 & \text{if~ $p/q\leqs(1-\veps) p^\star/q^\star$,} \\
    1 & \text{if~ $p/q\geqs(1+\veps) p^\star/q^\star$.}
  \end{cases}
 \end{align*}
Clearly, thresholds are not unique.
Indeed, if a threshold for a property $\QQQ$ is not sharp, then a constant multiple is also a threshold for $\QQQ$.
Sharp thresholds are not unique either, although a constant multiple of a sharp threshold for a property is not a threshold for that property.

A consequence of Proposition~\ref{propMonotoneAsymptEquiv} is that a threshold in $\cnp$ can be transferred to one in~$\cnm$:
\begin{prop}\label{propThresholdTransfer}
  Let $\QQQ$ be an increasing property that has a threshold $p^\star\geqs n^{-1}$ in $\cnp$.
  Then $np^\star/q^\star$ is a threshold for $\QQQ$ in $\cnm$.
\end{prop}
\begin{proof}
  Let $m^\star=np^\star/q^\star$.
  Suppose $m\gg m^\star$ and $p^+=m/(m+n)$, so $p^+/q^+\gg p^\star/q^\star$.
  Now, since $p^\star\gg n^{-1}$, we also have $p^+/q^+\gg\sqrt{p^+}/(q^+\sqrt{n})$, so we can find $\delta\gg\sqrt{p^+}/(q^+\sqrt{n})$ such that $p^+/q^+-\delta\gg p^\star/q^\star$.
  Since $\QQQ$ holds a.a.s. in $\cnp$ when $p/q\gg p^\star/q^\star$, by Proposition~\ref{propMonotoneAsymptEquiv}, $\QQQ$ also holds a.a.s. in $\cnm$.

  Similarly, suppose now that $m\ll m^\star$ and $p^-=m/(m+n)$, so $p^-/q^-\ll p^\star/q^\star$.
  Since $p^\star\gg n^{-1}$, we also have $p^\star/q^\star\gg\sqrt{p^-}/(q^-\sqrt{n})$, so we can find $\delta\gg\sqrt{p^-}/(q^-\sqrt{n})$ such that $p^-/q^-+\delta\ll p^\star/q^\star$.
  Since $\QQQ$ asymptotically almost never holds in $\cnp$ when $p/q\ll p^\star/q^\star$, then by Proposition~\ref{propMonotoneAsymptEquiv}, $\QQQ$ also asymptotically almost never holds in $\cnm$.
\end{proof}

To establish the presence of thresholds, we use the First Moment Method and the Second Moment Method.
The First Moment Method is an immediate corollary of Markov's Inequality and gives a sufficient condition for a property to asymptotically almost never hold.

\begin{prop}[{First Moment Method}]\label{propFirstMomentMethod}
  If $(X_n)_{n=1}^\infty$ is a sequence of nonnegative integer-valued random variables and
  $\expec{X_n}\ll1$,
  then
  $\prob{X_n = 0}\sim1$.
\end{prop}

The Second Moment Method, which follows from Chebyshev's Inequality, gives a sufficient condition for a property to hold a.a.s.
The following presentation follows~\cite[Section~4.3]{AS2016}.
Given an indexed set of events $\{A_i:{i\in I}\}$, we
write $i\sim j$ if $i\neq j$ and the events $A_i$ and $A_j$ are not independent.
We say that $A_i$ and $A_j$ are \emph{correlated}.
If $i\sim j$, we say that $i,j$ is a \emph{dependent} pair of indices.
For example, if, for each $i\in[n-1]$, the event $A_i$ occurs if the $i$th and $(i+1)$th terms of $\cnp$ are identical,
then $i\sim j$ precisely when $|i-j|=1$.

\begin{prop}[{Second Moment Method}]\label{propSecondMomentMethod}
  Suppose, for each $n\geqs1$, that $\{A_i:i\in I_n\}$ is a set of events.
  Suppose $X=X_n$ is the random variable that records how many of these events occur,
  and let $\Delta=\sum\limits_{i\sim j}\prob{A_i\wedge A_j}$, where the sum is over dependent pairs of indices.
  If
  $\expec{X}\gg1$
  and $\Delta\ll\expec{X}^2$, then
  $\prob{X>0}\sim1$.
\end{prop}

It is possible to determine the probability of a property holding \emph{at} its threshold.
To do this we use the Chen--Stein Method~\cite{Chen1975}.
The basic idea is that if events are mostly independent (for some properly defined notion of ``mostly''), then the number of these events that occur tends to a Poisson distribution.
As noted in~\cite{AGG1989}, under suitable conditions, Poisson convergence can be established by computing only the first and second moments.
In particular, this holds in the case of \mbox{\emph{dissociated}} events~\cite{BE1984,BE1987}, which is sufficient for our purposes. 
The following is adapted from~\cite[Theorem 4]{Janson1994}.
\begin{prop}[Chen--Stein Method]\label{propChenSteinMethod}
Suppose, for each $n\geqs1$, that $\{A_i:i\in I_n\}$ is a set of events, and that $|I_n|\gg1$.
Suppose $X_n$ is the random variable that records how many of these events occur,
and let
\[
\Delta \eq \sum_{i\sim j}\prob{A_i\wedge A_j}
\text{~~~~and~~~~}
\Lambda \eq \sum_{i\in I_n}\prob{A_i}^2 \:+\: \sum_{i\sim j} \prob{A_i}\prob{A_j} .
\]
If there exists a constant $\lambda>0$ such that
$\expec{X_n}\sim\lambda$,
and $\Delta+\Lambda\ll1$,
then $X_n$ converges in distribution to a Poisson distribution with mean $\lambda$.
In particular,
the asymptotic probability that none of the events occur is $e^{-\lambda}$.
\end{prop}


\section{Components and gaps}\label{sectComponents}

In this and the subsequent section we investigate how the structure of the random composition evolves as its size increases.
Initially, as long as $p\ll n^{-1}$, the expected number of nonzero terms in $\cnp$ equals $np\ll1$, so, by the First Moment Method, w.h.p. every term is zero and $\cnp$ is the empty $n$-permutation~$0^n$.
We are interested in what happens after this.

Our focus in this current section is on components and gaps.
A \emph{component} of a weak integer composition is a maximal run of nonzero terms.
A~\emph{gap} 
is a maximal run of zero terms.
For example, the composition in Figure~\ref{figComposition} on page~\pageref{figComposition} has 10 components, the longest having length~7.
It also has 10 gaps, the longest having length~4.

Components in $\cnp$ are equivalent to maximal runs of heads in sequences of coin tosses, where $p$ is the probability of a head.
For {constant} $p$, this has been a topic of study for many years.
In particular, the length of the longest run of heads has been investigated in considerable detail~\cite{ER1977,GO1980,GSW1986} (see also~\cite[pages 308--312]{FS2009}), of particular interest being the tiny fluctuations in its distribution that depend on the fractional part of $\log_2n$.
The asymptotic Gaussian distribution of the number of maximal runs of a fixed length (when $p$ is constant) is established in~\cite{MP2011}.
Finally,
from a statistical mechanics perspective,
Huillet~\cite{Huillet2011} investigates the length of both the shortest and longest component
in both the constant $p$ regime and also when $p\asymp n^{-1/k}$ for fixed $k\in \bbN$.

Components and gaps are dual in $\cnp$, in the sense that any statement about components can be transformed into one about gaps simply by switching the roles of $p$ (the probability that a term is nonzero) and $q$ (the probability that a term is zero).
Results concerning gaps thus follow directly from those concerning components.
Note however, that there is an asymmetry between $\cnp$ with small values of $p$ and $\cnp$ with small values of $q$.
Specifically, if $p\ll n^{-1}$, then w.h.p. there is literally nothing to see,
whereas if $q\ll n^{-1}$, then there's a lot of structure to investigate, with each term asymptotically having mean $1/q$ and variance~$1/q^2$.

Below we determine thresholds for the appearance and disappearance of components of a given length.
Initially, however, we have a brief look at the number of components in~$\cnp$.


\begin{prop}\label{propNumComps}
  In $\cnp$, the mean number of components  equals $nqp+p^2$, and
  the mean number of gaps equals $nqp+q^2$.
  Therefore, for any positive constant $\alpha$, asymptotically,
\[
\expec{\text{number of components / gaps in~} \cnp} \ssim
\begin{cases}
    0 \text{~~~/~~~} 1      , & \text{if~~} p\ll n^{-1} , \\
    \alpha \text{~~~/~~~} \alpha+1    , & \text{if~~} p\sim \alpha n^{-1} , \\
    np                  , & \text{if~~} n^{-1}\ll p\ll 1 , \\
    npq                 , & \text{if~~} p \text{~is constant} , \\
    nq                  , & \text{if~~} 1 \gg q \gg n^{-1} , \\
    \alpha+1 \text{~~~/~~~} \alpha    , & \text{if~~} q\sim \alpha n^{-1} , \\
    1 \text{~~~/~~~} 0      , & \text{if~~} n^{-1}\gg q .
\end{cases}
\]
\end{prop}
\begin{proof}
  We count the left ends of components. The probability that the $j$th term of $\cnp$ is the start of a component equals $p$ if $j=1$ and $qp$ if $2\leqs j\leqs n$.
Thus the expected number of components equals $p+(n-1)qp=nqp+p^2$.
\end{proof}

Thus the expected number of components is finite precisely when either $p\lesssim n^{-1}$ or $q\lesssim n^{-1}$.
In fact, for any fixed $k\geqs2$, we find
using the First and Second Moment Methods
that $p\asymp n^{-1}$ and $q\asymp n^{-1}$ are the lower and upper thresholds for there being at least $k$ components and at least $k$ gaps:
\begin{prop}\label{propManyComponents}
Suppose $k\geqs2$ is constant. Then,
\[
  \prob{\text{$\cnp$ has at least $k$ components / gaps}}  \ssim
  \begin{cases}
    0 , & \text{if~~} p\ll n^{-1} , \\
    1 , & \text{if~~} n^{-1} \ll p \text{~~and~~} q \gg n^{-1} , \\
    0 , & \text{if~~} n^{-1} \gg q .
  \end{cases}
\]
\end{prop}
\begin{proof}
  For each $i\in[n]$, let $B_i$ be the event that the $i$th term of $\cnp$ is the beginning of a component.
  Thus, $\prob{B_1}=p$, and $\prob{B_i}=qp$ if $i>1$.
  Suppose $\mathbf{i} := (i_1,i_2,\ldots,i_k)\in [n]^k$ is a vector such that $i_{j+1}\geqs i_j+2$ for each $j\in[k-1]$, and
  let $A_{\mathbf{i}}=B_{i_1}\wedge B_{i_2}\wedge \ldots\wedge B_{i_k}$.
  If $i_1=1$, then $\prob{A_{\mathbf{i}}}=q^{k-1}p^k$; otherwise $\prob{A_{\mathbf{i}}}=q^kp^k$.

  If $X$ is the total number of these $k$-tuples of components in $\cnp$, then by linearity of expectation, their expected number equals
  \[
  E_k \defeq
  \expec{X} \eq
  \binom{n-k}kq^kp^k \:+\: \binom{n-k}{k-1}q^{k-1}p^k
  \ssim \frac{n^kq^kp^k}{k!}\left(1+\frac{k}{nq}\right).
  \]
  Suppose $\omega\gg1$.
  If $p=n^{-1}/\omega$, then
  $
  \expec{X} \sim \omega^{-k}(1+k/n)/k! \ll 1 .
  $
  Thus, by the First Moment Method (Proposition~\ref{propFirstMomentMethod}), $X=0$ a.a.s., or equivalently, w.h.p. $\cnp$ has fewer than $k$ components.

  Similarly, if $q=n^{-1}/\omega$, then
  $
  \expec{X} \sim \omega^{-k}(1+k\omega)/k! \sll 1 ,
  $
  since $k\geqs2$, and so, again, w.h.p. $\cnp$ has fewer than $k$ components.

  Finally, suppose that $n^{-1}\ll p$ and $q\gg n^{-1}$.
  Then, $k/nq\ll1$ and $qp\gg n^{-1}$, so we have $\expec{X}=E_k\sim(nqp)^k/k!\gg1$.

  Distinct events $A_\mathbf{i}$ and $A_\mathbf{j}$ are correlated ($\mathbf{i}\sim \mathbf{j}$) if there exists a pair of indices $i_r$ in $\mathbf{i}$ and $j_s$ in $\mathbf{j}$ such that $|i_r-j_s|\leqs 1$.
  If, for any such pair, their difference equals 1, then $\prob{A_\mathbf{i}\wedge A_\mathbf{j}}=0$.
  Otherwise, the event $A_\mathbf{i}\wedge A_\mathbf{j}$ represents, for some $t\in[k-1]$, the presence of $k+t$ component left ends,
  with the indices of $k-t$ of these occurring in
  both $\mathbf{i}$ and~$\mathbf{j}$.
  Thus
  \[
            \Delta
  \defeq    \sum_{\mathbf{i}\sim \mathbf{j}} \prob{A_\mathbf{i}\wedge A_\mathbf{j}}
  \eq       \sum_{t=1}^{k-1} \binom{k+t}{k}\binom{k}{t} E_{k+t}
  \sless C_k \sum_{t=1}^{k-1} E_{k+t}
  \sless \frac{C_k}{k!} 
  \sum_{t=1}^{k-1} (nqp)^{k+t}
  ,
  \]
  for some constant $C_k$.
  Thus,
  \[
  \frac{\Delta}{\expec{X}^2} \sless C_k\, k! \sum_{t=1}^{k-1} (nqp)^{t-k} \sleqs C_k\, k! \sum_{s=1}^{k-1} (n/4)^{-s} \ll 1,
  \]
  since $qp\leqs\frac14$.

  So by the Second Moment Method (Proposition~\ref{propSecondMomentMethod}),
  if both $n^{-1} \ll p$ and $q \gg n^{-1}$ then
  $X>0$ a.a.s., or equivalently, w.h.p. $\cnp$ has at least $k$ components.
\end{proof}

\HIDE{
We now calculate the expected length of a component or gap in $\cnp$.

\begin{prop}\label{propCompLength}
  If $p>0$, then
  in $\cnp$
  the mean length of a component equals $\dfrac{n}{nq+p}$, and
  the mean length of a gap equals $\dfrac{n}{np+q}$.
Therefore, for any positive constant $\alpha$, asymptotically,
\begin{align*}
\expec{\text{length of a component in~} \cnp} &\ssim
\begin{cases}
    1       , & \text{if~~} p\ll 1 , \\
    1/q     , & \text{if~~} q\gg n^{-1} , \\
    n/(\alpha+1) , & \text{if~~} q\sim \alpha n^{-1} , \\
    n       , & \text{if~~} n^{-1}\gg q ,
\end{cases} \\[3pt]
\expec{\text{length of a gap in~} \cnp} &\ssim
\begin{cases}
    n       , & \text{if~~} p\ll n^{-1} , \\
    n/(\alpha+1) , & \text{if~~} p\sim \alpha n^{-1} , \\
    1/p     , & \text{if~~} n^{-1}\ll p , \\
    1       , & \text{if~~} 1\gg q .
\end{cases}
\end{align*}
\end{prop}
\begin{proof}
  The expected number of nonzero terms in $\cnp$ is $np$.
  So, by Proposition~\ref{propNumComps}, the expected component length equals $np/(nqp + p^2)=n/(nq+p)$.
\end{proof}

Note, in particular, that the expected length of a component asymptotically equals 1 as long as $p\ll 1$ and is asymptotically constant if $p\lesssim 1$.
} 

\subsection{The longest component and longest gap}\label{sectLongestComponent}

We now establish thresholds for $\cnp$ to have a component or gap exceeding a specified length.
Note that these are monotone properties, since increasing a term by one can never reduce the length of a component or increase the length of a gap.
If $C$ is a composition, let
$\cmax(C)$ be the length of the longest component of $C$, and
$\gmax(C)$ be the length of the longest gap in $C$.

Given some value of $k$,
for each $i\in[n+1-k]$,
let $A_i$ be the event ``$\cnp(i),\ldots,\cnp(i+k-1)$ are all nonzero''. Then $\prob{A_i}=p^k$.
So, if $X$ is the total number of runs of $k$ nonzero terms in $\cnp$, then by linearity of expectation, $\expec{X}=(n+1-k)p^k\sim np^k$, as long as $k\ll n$.

Distinct events $A_i$ and $A_j$ are correlated ($i\sim j$) if $|i-j|<k$.
If
$i\sim j$
and $i<j$, then $j=i+t$ for some $t\in[k-1]$, and $\prob{A_i\wedge A_j}=p^{k+t}$.
So,
  \[
           \Delta
  \defeq   \sum_{i\sim j} \prob{A_i\wedge A_j}
  \sless    np^k\sum_{t=1}^{k-1}p^t
  \sless    np^k\sum_{t=1}^{\infty}p^t
  \eq    np^{k+1}/q ,
  \]
and $\Delta/\expec{X}^2\lesssim p/np^kq$.
Moreover,
  \[
  \Lambda
  \defeq   \sum_i\prob{A_i}^2 
  \:+\:    \sum_{i\sim j} \prob{A_i}\prob{A_j} \sless nkp^{2k}.
  \]
To apply the Chen--Stein Method (Proposition~\ref{propChenSteinMethod}), it is sufficient to show that $\Delta\ll 1$ and $\Lambda\ll 1$.

The threshold for the appearance in $\cnp$ of a component of fixed length $k$ is $p\asymp n^{-1/k}$.
We also establish the probability of $\cnp$ having a component of length $k$ when $p\sim\alpha n^{-1/k}$:
\begin{prop}\label{propLongestComp1}
Suppose $k\geqs 1$ is constant. Then, for any positive constant $\alpha$,
  \begin{align*}
  \prob{\cmax(\cnp)\geqs k} & \ssim
  \begin{cases}
    0 ,             & \text{if~~} p\ll n^{-1/k} , \\
    1-e^{-\alpha^k} , & \text{if~~} p\sim \alpha n^{-1/k} , \\
    1 ,             & \text{if~~} n^{-1/k} \ll p,
  \end{cases}
  \\[3pt]
  \prob{\gmax(\cnp)\geqs k} & \ssim
  \begin{cases}
    1 ,             & \text{if~~} q\gg n^{-1/k} , \\
    1-e^{-\alpha^k} , & \text{if~~} q\sim \alpha n^{-1/k} , \\
    0 ,             & \text{if~~} n^{-1/k} \gg q.
  \end{cases}
  \end{align*}
\end{prop}
\begin{proof}
  If $p\ll n^{-1/k}$, then $\expec{X}\sim np^k\ll 1$, so by the First Moment Method, $X=0$ a.a.s., or equivalently, $\cmax(\cnp)< k$ a.a.s.

  If $n^{-1/k}\ll p$, then $\expec{X}\gg1$.
  If $p\ll 1$, then $\Delta/\expec{X}^2\lesssim p/np^k\ll p\ll1$.
  So by the Second Moment Method, $X>0$ a.a.s., or equivalently, $\cmax(\cnp)\geqs k$ a.a.s.
  Since the property of having a component of length at least $k$ is increasing, then by Proposition~\ref{propCnpIsMonotone} this also holds for larger~$p$.

  Finally, suppose that $p\sim\alpha n^{-1/k}$.
  Then $\expec{X}\sim \alpha^k$ and $\Delta < \alpha^kp/q\ll1$.
  Moreover, we have
  $\Lambda < nkp^{2k} \sim \alpha^{2k}kn^{-1} \ll 1$.
  So, by the Chen--Stein Method (Proposition~\ref{propChenSteinMethod}), the number of components in $\cnp$ of length~$k$ asymptotically satisfies a Poisson distribution with mean~$\alpha^k$.
  In particular, the asymptotic probability that no components have length $k$ or greater is $e^{-\alpha^k}$.
\end{proof}
Thus (using Proposition~\ref{propThresholdTransfer} to transfer the thresholds from $\cnp$ to $\cnm$),
as $m$ increases,
w.h.p. we first see components of length~2 in $\cnm$ when $m\asymp\sqrt{n}$, first see components of length~3 when $m\asymp n^{2/3}$, and so forth.
Hence, if $m\sim \alpha n^c$,
for positive constants $\alpha$ and $c\in(0,1)$,
then w.h.p. $\cmax(\cnm)$ takes only one value unless $c=1-1/k$ for some $k\in\bbN$ when a.a.s. it takes one of the two values in $\{k-1,k\}$.
However, once $m$ grows faster than $n^{1-\delta}$ for every $\delta>0$ (for example, $m=n/\log n$), a.a.s. the length of the longest component exceeds any fixed value.
Similarly, w.h.p. gaps of length 3 vanish once $m\gg n^{4/3}$, every gap has length 1 when $m\gg n^{3/2}$, and there are no gaps at all once $m\gg n^2$.

We now investigate the presence of components or gaps with lengths that increase with $n$.
Our first result reveals a sharp threshold:

\begin{prop}\label{propLongestComp2}
Suppose $1\ll k\ll \log n$. Then, for any $\omega\gg1$ and constant $\alpha$,
\begin{align*}
  \prob{\cmax(\cnp)\geqs k} & \ssim
  \begin{cases}
    0 ,                 & \text{if~~} p = e^{-(\log n+\omega)/k}, \\
    1-e^{-e^{\alpha}} , & \text{if~~} p = e^{-(\log n-\alpha)/k}, \\
    1 ,                 & \text{if~~} p = e^{-(\log n-\omega)/k},
  \end{cases}
  \\[3pt]
  \prob{\gmax(\cnp)\geqs k} & \ssim
  \begin{cases}
    1 ,                 & \text{if~~} q = e^{-(\log n-\omega)/k}, \\
    1-e^{-e^{\alpha}} , & \text{if~~} q = e^{-(\log n-\alpha)/k}, \\
    0 ,                 & \text{if~~} q = e^{-(\log n+\omega)/k}.
  \end{cases}
  \end{align*}
\end{prop}
\begin{proof}
  This proof, and many subsequent ones, follow the same bipartite or tripartite structure as seen above, so, from now on, we abbreviate the argument as far as possible.

  If $p = e^{-(\log n+\omega)/k}$, then $\expec{X}\sim e^{-\omega}\ll1$, so $\cmax(\cnp)< k$ a.a.s.

  If $p = e^{-(\log n-\omega)/k}$, then $\expec{X}\sim e^{\omega}\gg1$.
  If $\omega\ll\log n$, then $p\ll1$ and $\Delta/\expec{X}^2\lesssim pe^{-\omega}\ll 1$. So $\cmax(\cnp)\geqs k$ a.a.s.

  Finally, if $p \sim e^{-(\log n-\alpha)/k}$, then $\expec{X}\sim e^{\alpha}$ and $\Delta<pe^\alpha/q\ll1$. Moreover, we have
  $\Lambda < nkp^{2k} \sim e^{2\alpha}kn^{-1} \ll 1$.
  So, the number of components in $\cnp$ of length at least $k$ is asymptotically Poisson with mean~$e^\alpha$.
\end{proof}

We first see components of length the order of $\log n$ when $p$ is constant. Again, the threshold is sharp:
\begin{prop}\label{propLongestComp3}
Suppose $k= c\log n$ for some constant $c$. Then, for any $\omega\gg1$,
\begin{align*}
  \prob{\cmax(\cnp)\geqs k} & \ssim
  \begin{cases}
    0 ,                  & \text{if~~} p = e^{-1/c-\omega/\log n}, \\
    1 ,                  & \text{if~~} p = e^{-1/c+\omega/\log n},
  \end{cases}
  \\[3pt]
  \prob{\gmax(\cnp)\geqs k} & \ssim
  \begin{cases}
    1 ,                  & \text{if~~} q = e^{-1/c+\omega/\log n}, \\
    0 ,                  & \text{if~~} q = e^{-1/c-\omega/\log n}.
  \end{cases}
  \end{align*}
\end{prop}
\begin{proof}
  If $p = e^{-1/c-\omega/\log n}$, then $\expec{X}\sim e^{-c\omega}\ll1$, so $\cmax(\cnp)< k$ a.a.s.

  If $p = e^{-1/c+\omega/\log n}$, then $\expec{X}\sim e^{c\omega}\gg1$.
  If $\omega\ll\log n$, then $p$ is asymptotically constant and $\Delta/\expec{X}^2\lesssim pe^{-c\omega}/q\ll 1$. So $\cmax(\cnp)\geqs k$ a.a.s.
\end{proof}

It is notable (see~\cite[page 6]{Huillet2011}) that once $p$ has increased to a constant, w.h.p. the longest component in $\cnp$ has length of the order of $\log n$, despite the mean length of a component still being asymptotically constant (equal to $1/q$).

The threshold for the appearance of components of length $k\gg\log n$ is at $q=k^{-1}\log n$.
For example, w.h.p. we first see a component of length $\sqrt{n}$ in $\cnm$ once $m$ reaches $n^{3/2}/\log n$.

\begin{prop}\label{propLongestComp4}
Suppose $k= n^c$ for some $c\in(0,1)$. Then, for any $\omega\gg1$,
\begin{align*}
  \prob{\cmax(\cnp)\geqs k} & \ssim
  \begin{cases}
    0 ,                  & \text{if~~} q = k^{-1}(\log n+\omega), \\
    1 ,                  & \text{if~~} q = k^{-1}\big((1-c)\log n-\omega\big),
  \end{cases}
  \\[3pt]
  \prob{\gmax(\cnp)\geqs k} & \ssim
  \begin{cases}
    1 ,                  & \text{if~~} p = k^{-1}\big((1-c)\log n-\omega\big), \\
    0 ,                  & \text{if~~} p = k^{-1}(\log n+\omega) .
  \end{cases}
  \end{align*}
\end{prop}
\begin{proof}
  If $q = (\log n+\omega)/n^c$, then $\expec{X}\sim n\big(1-(\log n+\omega)/k\big)^k\sim e^{-\omega}\ll1$. Therefore, $\cmax(\cnp)< k$ a.a.s.

  We need to use an alternative bound on $\Delta$.
  \[
           \Delta
  \defeq   \sum_{i\sim j} \prob{A_i\wedge A_j}
  \sless    np^k\sum_{t=1}^{k-1}p^t
  \sless    nkp^{k+1} .
  \]
  Thus $\Delta/\expec{X}^2\lesssim pk/n(1-q)^k$.

  If $q = (\log n-\omega)/n^c$, then $\expec{X}\sim e^{\omega}\gg1$.
  If $q = \big((1-c)\log n-\omega\big)/n^c$, then
  \[
  \Delta/\expec{X}^2 \slsim
  pn^{c-1}\Big(1-\big((1-c)\log n-\omega\big)/k\Big)^{-k} \ssim
  pe^{-\omega} \ll 1 .
  \]
  So $\cmax(\cnp)\geqs k$ a.a.s.
\end{proof}

Finally, once $q\ll n^{-1}$, the expected number of zero terms in $\cnp$ equals $nq\ll1$, so w.h.p. every term is nonzero, and $\cnp$ consists of a single component of length $n$ with no gaps.


\subsection{The shortest component and shortest gap}\label{sectShortestComponent}

We now establish thresholds for $\cnp$ to have a component or gap shorter than a specified length.
Note that these are not monotone properties.
For example, adding one to the last term of the composition 11200 yields 11201 which reduces the length of the shortest component from 3 to~1.
If $C$ is a composition, let
$\cmin(C)$ be the length of the shortest component of~$C$, and
$\gmin(C)$ be the length of the shortest gap in $C$.

For each $\ell\in[n-1]$ and each $i\in[n+1-\ell]$, let $A_{i,\ell}$ be the event that the $i$th term of $\cnp$ is the start of a component of length $\ell$.
Then
\[
\prob{A_{i,\ell}} \eq
\begin{cases}
qp^\ell, & \text{if $i=1$ or $i=n+1-\ell$}, \\[2pt]
q^2p^\ell, & \text{otherwise}.
\end{cases}
\]
So, assuming $\ell\ll n$, if $X_\ell$ is the number of components of length $\ell$ in $\cnp$, then
\[
\expec{X_\ell} \eq (n-1-\ell)q^2p^\ell+2qp^\ell \ssim nq^2p^\ell .
\]
Given some $k\ll n$ and assuming $kq\ll1$ (so $1-p^k\sim kq$), let $X$ be the total number of components of length at most $k$ in $\cnp$.
Then
\[
\expec{X} \eq \sum_{\ell=1}^k \expec{X_\ell} \ssim nq^2(p+p^2+\ldots+p^k) \eq npq(1-p^k)
\ssim knq^2 .
\]
Distinct events $A_{i,r}$ and $A_{j,s}$ ($i\leqs j$) are correlated ($i,r\sim j,s$) in two situations.
If $j\leqs i+r$, then the corresponding components overlap in a contradictory manner, so $\prob{A_{i,r}\wedge A_{j,s}}=0$.
If $j= i+r+1$, then the corresponding components are separated by a single zero term and $\prob{A_{i,r}\wedge A_{j,s}}=q^3p^{r+s}$,
except when the pair of components occur at the start or end of the composition, in which case $\prob{A_{i,r}\wedge A_{j,s}}=q^2p^{r+s}$.
Thus,
\begin{align*}
\Delta \defeq
\sum_{i,r\,\sim\, j,s}
\prob{A_{i,r}\wedge A_{j,s}}
&\eq \sum_{r=1}^k \sum_{s=1}^k (n-2-r-s)q^3 p^{r+s}+2 q^2 p^{r+s}
\\
&\ssim n\sum_{r=1}^k \sum_{s=1}^kq^3 p^{r+s}
\eq np^2q\big(1-p^k\big)^2
\ssim k^2nq^3.
\end{align*}
Thus $\Delta/\expec{X}^2 \sim 1/nq$, which tends to zero as long as $q\gg n^{-1}$.
Moreover,
\begin{align*}
\Lambda
& \defeq
\sum_{i,\ell}
\prob{A_{i,\ell}}^2 \:+\:
\sum_{i,r\,\sim\, j,s}
\prob{A_{i,r}}\prob{A_{j,s}} \\
& \ssim n \sum_{r=1}^k (r+2) \sum_{s=1}^k q^4p^{r+s}
\ssim \frac12 k^2(k+5) n q^4 \slsim k^3n q^4 .
\end{align*}

Our first result shows thresholds at $p\asymp n^{-1/2}$, for the appearance of short gaps, and at $q\asymp n^{-1/2}$ for the disappearance of short components:
\begin{prop}\label{propShortestComp1}
Suppose $k\geqs 1$ is constant. Then, for any positive constant $\alpha$,
  \begin{align*}
  \prob{\cmin(\cnp)> k} & \ssim
  \begin{cases}
    0 ,              & \text{if~~} q\gg n^{-1/2} , \\
    e^{-\alpha^2 k} , & \text{if~~} q\sim \alpha n^{-1/2} , \\
    1 ,              & \text{if~~} n^{-1/2} \gg q ,
  \end{cases}
  \\[3pt]
  \prob{\gmin(\cnp)> k} & \ssim
  \begin{cases}
    1 ,              & \text{if~~} p \ll n^{-1/2} , \\
    e^{-\alpha^2 k} , & \text{if~~} p\sim \alpha n^{-1/2} , \\
    0 ,              & \text{if~~} n^{-1/2} \ll p .
  \end{cases}
  \end{align*}
\end{prop}
\begin{proof}
  Suppose $\omega\gg1$. If $q=n^{-1/2}/\omega$, then $\expec{X}\sim k/\omega^2\ll1$, so w.h.p. there are no components of length $k$ or less, and $\cmin(\cnp)> k$ a.a.s.

  If $q = n^{-1/2}\omega$, then $\expec{X}\sim k\omega^2\gg1$.
  Since $\Delta/\expec{X}^2 \ll1$, we have $\cmin(\cnp)\leqs k$ a.a.s.

  Finally, if $q \sim \alpha n^{-1/2}$, then $\expec{X}\sim \alpha^2 k$.
  Also, $\Delta \sim \alpha^3k^2 n^{-1/2}\ll1$ and $\Lambda \lesssim \alpha^4k^3 n^{-1/2}\ll1$.
  So, the number of components in $\cnp$ of length at most $k$ is asymptotically Poisson with mean~$\alpha^2 k$.
  In particular, the asymptotic probability that no components have length $k$ or less is $e^{-\alpha^2k}$.
\end{proof}

Thus, as soon as $p\gg n^{-1/2}$, w.h.p. there is a gap of length 1 in $\cnp$. Prior to this, a.a.s. there is no gap of \emph{any} fixed length.
However, w.h.p. components of length 1 remain until $q\asymp n^{-1/2}$.
At this point, the \emph{longest} components have length the order of $\sqrt{n}\log n$ (Proposition~\ref{propLongestComp4}).
But once $q\ll n^{-1/2}$, a.a.s. no component of \emph{any} fixed length remains.
At these thresholds, we can calculate the asymptotic probability that the shortest component or gap has a given length.
For instance, when $p\sim n^{-1/2}$, with probability $1-e^{-1}\approx 0.63$ a gap of length 1 has already appeared.

The thresholds for the disappearance of longer components are as follows.
In particular, components of length $n^c$ vanish a.a.s. once $q\ll n^{-(1+c)/2}$.
For example, once $m\gg n^{7/4}$, a.a.s. all components have length $\gg\sqrt{n}$.

\begin{prop}\label{propShortestComp2}
Suppose $1\ll k\ll n$. Then, for any positive constant $\alpha$,
  \begin{align*}
  \prob{\cmin(\cnp)> k} & \ssim
  \begin{cases}
    0 ,               & \text{if~~} q\gg 1/\sqrt{kn} , \\
    e^{-\alpha^2}   , & \text{if~~} q\sim \alpha /\sqrt{kn} , \\
    1 ,               & \text{if~~} 1/\sqrt{kn} \gg q ,
  \end{cases}
  \\[3pt]
  \prob{\gmin(\cnp)> k} & \ssim
  \begin{cases}
    1 ,               & \text{if~~} p \ll 1/\sqrt{kn} , \\
    e^{-\alpha^2}   , & \text{if~~} p\sim \alpha /\sqrt{kn} , \\
    0 ,               & \text{if~~} 1/\sqrt{kn} \ll p .
  \end{cases}
  \end{align*}
\end{prop}
\begin{proof}
  First, note that $kq\sim\sqrt{k/n}\ll1$, as required for our asymptotics to be valid.

  Suppose $\omega\gg1$. If $q=\omega^{-1}/\sqrt{kn}$, then $\expec{X}\sim \omega^{-2}\ll1$, so w.h.p. $\cmin(\cnp)> k$.

  If $q = \omega/\sqrt{kn}$, then $\expec{X}\sim \omega^2\gg1$.
  Since $\Delta/\expec{X}^2 \ll1$, we have $\cmin(\cnp)\leqs k$ a.a.s.

  Finally, if $q \sim \alpha/\sqrt{kn}$, then $\expec{X}\sim \alpha^2$.
  Also, $\Delta \sim \alpha^3 \sqrt{k/n}\ll1$ and $\Lambda \lesssim \alpha^4k/n\ll1$.
  So, the number of components in $\cnp$ of length at most $k$ is asymptotically Poisson with mean~$\alpha^2$.
\end{proof}

\section{Patterns}\label{sectPatterns}

The focus of the remainder of this work is 
on the appearance and disappearance of various types of \emph{pattern}.
A pattern is simply a sub-composition, under some notion of \emph{containment}.
Below, we consider a variety of different types of pattern containment.
To distinguish between these, we sometimes add a prefix to the pattern (e.g., $\ep$ or $\gp$). Definitions are given below.
For the most part we require the terms of a pattern to occur consecutively in a composition, which we signify by using an overline (e.g., $\epc[123]$).
This requirement is eventually relaxed
in the final subsection. 

We begin in Section~\ref{sectExactPatts} by considering \emph{exact} patterns, in which terms must take specified values, including an investigation of runs of equal nonzero terms and of \emph{square} patterns (runs of $k$ terms equal to~$k$).
Section~\ref{sectMonotonePatts} concerns \emph{upper} and \emph{lower} patterns, including a consideration of the largest and smallest terms.
We then look at patterns specifying the relative ordering of terms (Section~\ref{sectOrderingPatts}), including determining the threshold for $\cnp$ to be a
\emph{Carlitz} composition (having no adjacent pair of equal terms).
Finally, in Section~\ref{sectNonConsecPatts}, we consider nonconsecutive patterns, including determining the threshold for all the terms of $\cnp$ to be distinct.

There is quite an extensive literature on patterns in compositions and words, some of which is mentioned below.
This includes comprehensive expositions by Heubach and Mansour~\cite{HM2010} and Kitaev~\cite{Kitaev2011}.
However, 
the enumerative and generatingfunctionological perspective 
taken in these works is somewhat orthogonal to our interests.

\subsection{Exact consecutive patterns}\label{sectExactPatts}

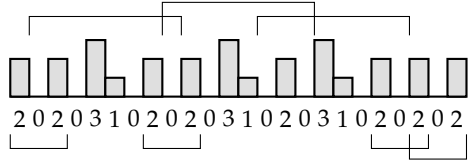
\begin{figure}[t]
\begin{center}
\begin{tikzpicture}[scale=0.25]
  \plotc{24}{2,0,2,0,3,1,0,2,0,2,0,3,1,0,2,0,3,1,0,2,0,2,0,2}
  \draw[] (1,3.25) -- (1,4.25) -- (9,4.25) -- (9,3.25);
  \draw[] (8,3.25) -- (8,5) -- (16,5) -- (16,3.25);
  \draw[] (13,3.25) -- (13,4.25) -- (21,4.25) -- (21,3.25);
  \draw[] (0,-2) -- (0,-2.75) -- (3,-2.75) -- (3,-2);
  \draw[] (7,-2) -- (7,-2.75) -- (10,-2.75) -- (10,-2);
  \draw[] (19,-2) -- (19,-2.75) -- (22,-2.75) -- (22,-2);
  \draw[] (21,-2) -- (21,-3.3125) -- (24,-3.3125) -- (24,-2);
\end{tikzpicture}
\end{center}
\vspace{-6pt}
\caption{A composition containing
four occurrences of the exact consecutive pattern $\epc[202]$
and three occurrences of $\epc[02031020]$}\label{figPatterns}
\end{figure}

We begin with the simplest notion of pattern.
The \emph{exact consecutive pattern}
$\epc[r_1\ldots r_k]$ occurs at position $i$ in a composition $C$ if, for each $j\in[k]$,
we have $C(i-1+j)=r_j$.
In the language of combinatorics on words, such a pattern occurs in a composition if it is a \emph{factor} of the composition.
See Figure~\ref{figPatterns} for an illustration.
A pattern is \emph{nonzero} if at least one of its terms is positive.

The presence of an exact pattern is not a monotone property.
It isn't even convex;
for example, 
$\epc[22]$ occurs in the compositions $221$ and $322$, but does not occur in $321$.
However, in $\cnp$, the presence of a nonzero exact pattern $\epc$ does exhibit both a lower threshold, for its appearance, and 
an upper threshold, for its disappearance.
If $\pi=r_1\ldots r_k$, then
the lower threshold depends only on the \emph{size} $|\pi|=\sum_{i=1}^kr_i$ of the pattern,
whereas the upper threshold depends only on its \emph{length}~$k$.
(Recall that in this work $|\pi|$ always denotes the sum of the terms of $\pi$, \emph{not} its length.)

\begin{prop}\label{propExactPattConsec}
If $\epc$ is a nonzero exact consecutive pattern of length $k$, then for any positive constant~$\alpha$,
\[
  \prob{\text{$\cnp$ contains $\epc$}} \ssim
  \begin{cases}
    0, & \text{if~ $p\ll n^{-1/|\pi|}$},  \\
    1-e^{-\alpha^{|\pi|}}, & \text{if~ $p\sim\alpha n^{-1/|\pi|}$},  \\
    1, & \text{if~ $n^{-1/|\pi|}\ll p$ ~and~ $q\gg n^{-1/k}$},  \\
    1-e^{-\alpha^k}, & \text{if~ $q\sim\alpha n^{-1/k}$},  \\
    0, & \text{if~ $n^{-1/k}\gg q$} .
  \end{cases}
\]
The expected number of occurrences of $\epc$ in $\cnp$ is maximal when $p/q=|\pi|/k$.
\end{prop}

\begin{proof}
Suppose $\pi=r_1\ldots r_k$ and $|\pi|=s$.
For each $i\in[n+1-k]$, let $A_i$ be the event that $\epc$ occurs at position $i$ in $\cnp$, and let $X$ be the number of occurrences of $\epc$ in $\cnp$.
Then,
$
\prob{A_i}
\eq q^k p^s,
$
and 
$\expec{X}\sim nq^k p^s$, which,
by elementary calculus, is seen to be maximal when $p=s/(k+s)$.

  If~$p\ll n^{-1/s}$, then $\expec{X}\sim np^s\ll 1$.
  Similarly, if~$q\ll n^{-1/k}$, then $\expec{X}\sim nq^k\ll 1$.
  Thus, by the First Moment Method, in either case, w.h.p. $\epc$ doesn't occur in~$\cnp$.

  Distinct events $A_i$ and $A_j$ are correlated if $t=|j-i|<k$.
  If $r_\ell\neq r_{\ell+t}$ for some $\ell\in[k-t]$, then $\prob{A_i\wedge A_j}=0$.
  Otherwise, $\prob{A_i\wedge A_j} \leqs q^{k+1}p^{s+1}$, since $\pi$ is nonzero.
 Thus,
 \[
 \Delta \defeq \sum_{i\sim j}\prob{A_i\wedge A_j} \sleqs nk q^{k+1}p^{s+1}
 \text{~~~~~and~~~~~}
 R \defeq {\Delta}/{\expec{X}^2} \slsim \frac{k}{{nq^{k-1}p^{s-1}}} .
 \]
 Moreover,
  $
  \Lambda
  \defeq
  \sum_i\prob{A_i}^2
  +
  \sum_{i\sim j} \prob{A_i}\prob{A_j}
  \ssim
  nkp^{2s}q^{2k}
  .
  $

  Suppose $p= \omega n^{-1/s}\ll1$ for some $\omega\gg1$. Then 
  $\expec{X}\sim\omega^s\gg1$
  and $R\lesssim kp/\omega^s\ll1$.
  Similarly, if $q= \omega n^{-1/k}\ll1$ for some $\omega\gg1$, then 
  $\expec{X}\sim\omega^k\gg1$ and
  $R\lesssim kq/\omega^k\ll1$.
  Finally, if $p$ is asymptotically bounded away from both $0$ and $1$,
  then $\expec{X}\asymp n\gg 1$
  and $R\asymp n^{-1}\ll 1$.
  Hence, by the Second Moment Method, if $n^{-1/s}\ll p$ and $q\gg n^{-1/k}$, w.h.p. $\epc$ occurs in $\cnp$.

  Suppose $p\sim\alpha n^{-1/s}$. Then $\expec{X}\sim \alpha^s$ and $\Delta\leqs \alpha^s kp\ll 1$, and $\Lambda\sim \alpha^{2s}k/n\ll1$.
  So, by the Chen--Stein Method, the number of occurrences of $\epc$ is asymptotically Poisson with mean~$\alpha^s$.
  Similarly, if $q=\alpha n^{-1/k}$ then $\expec{X}\sim \alpha^k$ and $\Delta\leqs \alpha^k k q\ll 1$, and $\Lambda\sim \alpha^{2k}k/n\ll1$, so
  the number of occurrences of $\epc$ is asymptotically Poisson with mean~$\alpha^k$.
\end{proof}

Thus, we have a dichotomy between the arrival and the departure of exact consecutive patterns as $\cnp$ evolves,
the former being ordered by size, smaller patterns appearing before larger ones,
and the latter being ordered by length, longer patterns disappearing before shorter ones.

The following proposition enables us to transfer the thresholds for exact consecutive patterns from $\cnp$ to $\cnm$.
Since the relevant properties are not monotone, we can't simply apply Proposition~\ref{propMonotoneAsymptEquiv}.

\begin{prop}\label{propExactPattCnmEqCnpProb}
  If $\epc$ is an exact consecutive pattern
  and $m\sim np/q\gg1$, then
  \[
  \prob{\text{$\cnm$ contains $\epc$}}
  \ssim
  \prob{\text{$\cnp$ contains $\epc$}}
  .
  \]
\end{prop}
\begin{proof}
  Suppose $\pi$ has length $k$ and size $s$.
  For each $i\in[n+1-k]$, let $P_i$ be the probability that $\epc$ occurs at position $i$ in $\cnm$. Then,
  \[
  P_i \eq \binom{m-s+n-k-1}{m-s} \binom{m+n-1}{m}^{\!-1} .
  \]
  For brevity, let $n_1=n-1$, $n_k=n_1-k$ and $m_s=m-s$.
  Note that $n_1\sim n_k\sim n$ and $m_s\sim m$
  (as along as $k\ll n$ and $s\ll m$),
  and
  also that $p\sim m/(m+n)$ and $q\sim n/(m+n)$.

  Then, by Stirling's approximation,
  \begin{align*}
  P_i & \eq \binom{m_s+n_k}{m_s} \binom{m+n_1}{m}^{\!-1} \\[6pt]
  & \ssim \sqrt{\frac{(m_s+n_k)\,m\,n_1}{m_s\,n_k\,(m+n_1)}} \,
  \frac{(m_s+n_k)^{m_s+n_k}}{m_s^{m_s} \, n_k^{n_k}} \, \frac{m^m \, n_1^{n_1}}{(m+n_1)^{m+n_1}} \\[6pt]
  & \ssim
  \Big(\frac{m_s+n_k}{m+n_1}\Big)^{\!m+n_1} (m_s+n_k)^{-s-k} \,
  \Big(\frac{m}{m_s}\Big)^{\!m_s} m^s \,
  \Big(\frac{n_1}{n_k}\Big)^{\!n_k} n_1^{\,k} \\[6pt]
  & \eq
  \Big(1-\frac{s+k}{m+n_1}\Big)^{\!m+n_1}
  \Big(1+\frac{s}{m_s}\Big)^{\!m_s}
  \Big(1+\frac{k}{n_k}\Big)^{\!n_k}   \,
  \frac{m^s\,n_1^{\,k}}{(m_s+n_k)^{s+k}} \\[6pt]
  & \ssim e^{-s-k} \, e^s \, e^k \, p^s \, q^k 
  \eq p^s q^k \ssim \prob{\text{$\epc$ occurs at position $i$ in $\cnp$}} .
  \end{align*}
  The result then follows from the fact that the probability of $\cnp$ or $\cnm$ containing an exact consecutive pattern depends only on the asymptotic probabilities of exact consecutive patterns occurring at a given position (see the proof of Proposition~\ref{propExactPattConsec}).
\end{proof}

Thus, $m\asymp n^{1-1/s}$ is the threshold for the appearance of each exact consecutive pattern of size $s$ in $\cnm$.
So, if $0<\gamma<1$ and $m\sim n^\gamma$, then w.h.p. $\cnm$ contains every exact consecutive pattern of size less than $1/(1-\gamma)$, but contains no such pattern of size greater than $1/(1-\gamma)$.
For example, every fixed length gap of the form $\epc[10\ldots01]$ appears when $m\asymp n^{1/2}$ (see Proposition~\ref{propShortestComp1}), whereas no gap delimited by larger terms appears before $m\asymp n^{2/3}$.

Similarly, $m\asymp n^{1+1/k}$ is the threshold for the disappearance of each exact consecutive pattern of length $k$ in $\cnm$.
So, if $1<\gamma<2$ and $m\sim n^\gamma$, then w.h.p. $\cnm$ contains every exact consecutive pattern of length less than $1/(\gamma-1)$, but contains no such pattern of length greater than $1/(\gamma-1)$.

These results establish when any given exact consecutive pattern is present.
For example, w.h.p. the pattern $\epc[314159]$ appears when $m\asymp n^{22/23}$ and has disappeared once \mbox{$m\gg n^{7/6}$}.
If $\pi_1$ is both shorter and smaller than $\pi_2$, then $\epc[\pi_1]$ arrives before $\epc[\pi_2]$ and leaves after~$\epc[\pi_2]$.
However, since the arrival and departure of patterns depend on different parameters, the values of which can be chosen independently, it is possible for the departure order of a number of patterns to be any permutation of their arrival order.
For example, w.h.p., $\epc[101]$, $\epc[3]$, $\epc[1111]$ and $\epc[23]$ arrive in that order, but depart in the order $\epc[1111]$, $\epc[101]$, $\epc[23]$, $\epc[3]$.

Observe that there is a range of values of $m$ for which, a.a.s., \emph{any} given exact consecutive pattern is present in~$\cnm$.
Specifically, if 
$n^{1-\delta}\ll m\ll n^{1+\delta}$ 
for every $\delta>0$ (for example, $m\sim n/\log n$ or $m\sim n\log n$), then {any} given pattern occurs w.h.p.

In contrast, once $m\gg n^2$, then w.h.p. any specific exact consecutive pattern is absent from $\cnm$.
In particular, this holds for patterns of length one: given any value $r$, a.a.s. no term of $\cnm$ is equal to~$r$.

\subsubsection{Runs of equal terms}\label{sectEqualRuns}

In this section
we investigate runs of $k$ equal nonzero terms; that is, the occurrence of any of the exact consecutive patterns $\epc[11\ldots1]$, $\epc[22\ldots2]$, $\epc[33\ldots3]$, etc. of length $k$.

Given $k\in[n]$ and $i\in[n+1-k]$, let $A_i$ be the event that the $i$th term of $\cnp$ is the start of a run of $k$ equal nonzero terms.
Then
\[
\prob{A_i} \eq \sum_{r=1}^\infty (qp^r)^k \eq \frac{p^kq^k}{1-p^k}
.
\]
Thus, if $X$ is the number of runs of $k$ equal nonzero terms, $\expec{X}\sim n p^kq^k/(1-p^k)$, assuming $k\ll n$.
Distinct events $A_i$ and $A_j$ are correlated if $t=|j-i|<k$, with
\[
\prob{A_i\wedge A_j} \eq \sum_{r=1}^\infty (qp^r)^{k+t} \eq \frac{(pq)^{k+t}}{1-p^{k+t}}
\sleqs \frac{(pq)^{k+1}}{1-p^{k+1}}.
\]

\begin{prop}\label{propEqualRuns1}
  Let
  $P_k=\liminfty\prob{\text{$\cnp$ contains a run of $k$ equal nonzero terms}}$.
  Then, for any fixed $k\geqs2$ and positive constant $\alpha$,
\[
  P_k \eq
  \begin{cases}
    0, & \text{if~ $p\ll n^{-1/k}$},  \\
    1-e^{-\alpha^{k}}, & \text{if~ $p\sim\alpha n^{-1/k}$},  \\
    1, & \text{if~ $n^{-1/k}\ll p$ ~and~ $q\gg n^{-1/(k-1)}$},  \\
    1-e^{-\alpha^{k-1}/k}, & \text{if~ $q\sim\alpha n^{-1/(k-1)}$},  \\
    0, & \text{if~ $n^{-1/(k-1)}\gg q$} .
  \end{cases}
\]
\end{prop}
\begin{proof}
  Suppose $\omega\gg 1$.
  If $p= n^{-1/k}/\omega$, then $\expec{X}\sim\omega^{-k}\ll 1$.
  Similarly, if $q= n^{-1/(k-1)}/\omega$, then $\expec{X}\sim\omega^{-(k-1)}/k\ll 1$.
  So, in either case, w.h.p. $\cnp$ has no run of $k$ equal nonzero terms.

Let
\[
\Delta \defeq \sum_{i\sim j}\prob{A_i\wedge A_j} \sleqs \frac{nk(pq)^{k+1}}{1-p^{k+1}}
\text{~~~~~and~~~~~}
R \defeq \frac{\Delta}{\expec{X}^2} \slsim \frac{kpq(1-p^k)^2}{np^kq^k(1-p^{k+1})} .
\]
  If $p= \omega n^{-1/k}\ll 1$, then $\expec{X}\sim\omega^k\gg 1$ and $R\lesssim kp\omega^{-k}\ll1$.
  Similarly, if $q= \omega n^{-1/(k-1)}\ll 1$, then $\expec{X}\sim\omega^{k-1}/k\gg 1$ and $R\lesssim k^2q\omega^{-(k-1)}\ll1$.
  Finally, if $p$ is asymptotically bounded away from both $0$ and $1$, then $\expec{X}\asymp n\gg 1$ and $R\asymp n^{-1}\ll1$.
  Hence, if $n^{-1/s}\ll p$ and $q\gg n^{-1/(k-1)}$, w.h.p. $\cnp$ contains a run of $k$ equal nonzero terms.

  Let
  \[
  \Lambda
  \defeq
  \sum_i\prob{A_i}^2 \:+\: \sum_{i\sim j} \prob{A_i}\prob{A_j}
  \ssim
  \frac{nk(pq)^{2k}}{(1-p^k)^2}
  .
  \]
  Suppose $p\sim\alpha n^{-1/k}$. Then $\expec{X}\sim \alpha^k$ and $\Delta\leqs k\alpha^{k+1}n^{-1/k}\ll 1$, and $\Lambda\sim k\alpha^{2k}/n\ll1$.
  So the number of occurrences of runs of $k$ equal nonzero terms is asymptotically Poisson with mean $\alpha^k$.
  Similarly, if $q\sim\alpha n^{-1/(k-1)}$ then $\expec{X}\sim \alpha^{k-1}/k$ and $\Delta\leqs \alpha^kn^{-1/(k-1)}\ll 1$, and $\Lambda\sim \alpha^{2k-2}/kn\ll1$, so
  the number of occurrences of runs of $k$ equal nonzero terms is asymptotically Poisson with mean~$\alpha^{k-1}/k$.
\end{proof}

Thus, as one would expect, a.a.s the arrival of a run of $k$ nonzero terms coincides with the appearance of the length $k$ pattern $\epc[11\ldots1]$.
However, the threshold at $q\asymp n^{-1/(k-1)}$ for the departure of all runs of $k$ nonzero terms is later than that at $q\asymp n^{-1/k}$ for the disappearance of any specific length $k$ pattern $\epc[rr\ldots r]$ (see Proposition~\ref{propExactPattConsec}).

In the case that $p$ is constant,
the length of the {longest} run in $\cnp$ has been investigated in detail (see~\cite{Louchard2002,GKP2003,Eryilmaz2006}),
and Gafni~\cite{Gafni2015} has considered the question for a random composition with positive terms.
For our variant of the problem, we have the following sharp bounds on the length of the longest run of nonzero terms, which is maximal when $p=\frac12$.

\begin{prop}\label{propEqualRuns2}
If $p\in(0,1)$ is constant and $\gamma=1/pq$, then for any $\omega\gg1$,
\[
\prob{\text{$\cnp$ contains a run of $k$ equal nonzero terms}} \ssim
\begin{cases}
0 & \text{if $k=\log_\gamma n+\omega$}, \\[3pt]
1 & \text{if $k=\log_\gamma n-\omega$}.
\end{cases}
\]
\end{prop}
\begin{proof}
  If $k=\log_\gamma n+\omega$, then $\expec{X}\sim np^kq^k\sim \gamma^{-\omega}\ll1$ since $\gamma>1$. So w.h.p. $\cnp$ has no run of $k$ equal nonzero terms.

  To apply the Second Moment Method, we require a stronger bound on $\Delta$ than that used in Proposition~\ref{propEqualRuns1}:
  \[
  \Delta \defeq \sum_{i\sim j}\prob{A_i\wedge A_j} \sleqs n\sum_{t=1}^\infty \frac{(pq)^{k+t}}{1-p^{k+1}}
  \ssim \frac{n(pq)^{k+1}}{1-pq} .
  \]
  If $k=\log_\gamma n-\omega$, then $\expec{X} \sim \gamma^\omega\gg1$, so
  \[
  \frac{\Delta}{\expec{X}^2} \slsim \frac{(pq)^{-(k-1)}}{n(1-pq)} \ssim \frac{\gamma^{-\omega}}{\gamma-1} \sll 1,
  \]
  and w.h.p. $\cnp$ contains a run of $k$ equal nonzero terms.
\end{proof}

\subsubsection{Square patterns}\label{sectSquarePatts}

We conclude our investigation of exact consecutive patterns by considering the $k$-\emph{square} pattern
$\epc[kk\ldots k]$, of length $k$,
an occurrence of which is a run of $k$ terms each equal to $k$.
For example, the composition in Figure~\ref{figComposition} on page~\pageref{figComposition} contains a 4-square and two 2-squares.
From Proposition~\ref{propExactPattConsec}, we know that, for any fixed $k$, a.a.s. $k$-squares appear when $p\asymp n^{-1/k^2}$ and disappear when $q\asymp n^{-1/k}$.

What is the \emph{largest} square that we can expect to see in the evolution of the random composition?
Here we show that w.h.p. it has side length a little greater than ${\log n}/{\log\log n}$, as expected from the following non-rigorous heuristic argument:
We expect to see the largest $k$-square when the two thresholds coalesce, that is when both $p=n^{-1/k^2}$ and $q=n^{-1/k}$, so $q=p^k$.
Thus $p^k+p=1$ or $(1-q)^k+(1-q)=1$. Approximating $(1-q)^k$ with $1-kq$ (since $k\gg1$ implies $q\ll1$) then yields $q=1/(k+1)$.
So $1/(k+1)=n^{-1/k}$ or $k\log(k+1)=\log n$, whose asymptotic solution matches the value of $k$ in the following proposition.

\begin{prop}\label{propSquarePat}
Suppose $q= \dfrac{\theta\log\log n}{\log n}$
and $k=\dfrac{\log n}{\log\log n}\left( 1+ \dfrac{c \log\log\log n}{\log\log n} \right)$.
Then,
\[
\prob{\text{$\cnp$ contains a $k$-square pattern}} \ssim
\begin{cases}
0 & \text{if $c\geqs1$ or $\theta$ is far}, \\
1 & \text{if $c<1$ and $\theta$ is near},
\end{cases}
\]
where we say that $\theta$ is \emph{near} if $(\log\log n)^{-\gamma} \ll \theta \ll \log\log\log n$ for all constant $\gamma>0$,
and $\theta$ is \emph{far} if either $\theta \ll (\log\log n)^{-\gamma}$ for all constant $\gamma$ or else $\log\log\log n\ll \theta$.
\end{prop}

\begin{proof}
For $i\in[n+1-k]$, let $A_i$ be the event that the $i$th term of $\cnp$ is the start of a $k$-square.
Then
$\prob{A_i} \eq q^kp^{k^2}.$
Thus, if $X$ is the number of $k$-squares,
\[
\expec{X} \ssim q^kp^{k^2} \ssim \left(\frac{\theta\log\log n}{\log n}\right)^{\!k}\exp\left( -\frac{k^2\,\theta \log\log n}{\log n} \right)
.
\]
Substituting for $k$ and taking the dominant term, if $\theta$ is near and $c\neq1$ then
\[
\log\expec{X} \ssim (1-c) \frac{\log n \log \log \log n}{\log \log n}
,
\]
and if $\theta$ is near and $c=1$ then
\[
\log\expec{X} \ssim - \frac{\log n}{\log \log n}
.
\]
If $\omega\gg1$ and either $\theta=(\log\log n)^{-\omega}$ or else $\theta=\omega\log\log\log n$, then
\[
\log\expec{X} \ssim - \frac{\omega \log n \log \log \log n}{\log \log n}
.
\]
Thus, if $c\geqs1$ or $\theta$ is far, we have $\expec{X}\ll1$ and so w.h.p. $\cnp$ contains no $k$-square, whereas if $c<1$ and $\theta$ is near then $\expec{X}\gg1$.

Distinct events $A_i$ and $A_j$ are correlated if $t=|j-i|<k$, with
$
\prob{A_i\wedge A_j} \eq q^{k+t}p^{k^2+k t}
$.
Thus,
\[
\Delta \defeq \sum_{i\sim j}\prob{A_i\wedge A_j} \sleqs n\sum_{t=1}^\infty q^{k+t}p^{k^2+k t}
  \eq \frac{n q^{k+1}p^{k^2+k}}{1-q p^k}
  \ssim n q^{k+1}p^{k^2+k} ,
\]
and
\[
R\defeq
\frac{\Delta}{\expec{X}^2}
\slsim n^{-1} q^{1-k}p^{k-k^2}
\ssim n^{-1}
\left(\frac{\log n}{\log \log n}\right)^{\!k-1} \exp \left(
(k^2-k)
\frac{\log\log n}{\log n}
\right).
\]
Substituting for $k$ and taking the dominant term, if $\theta$ is near and $c\neq1$ then
\[
\log R \ssim (c-1) \frac{\log n \log \log \log n}{\log \log n}
.
\]
So if $c<1$, we have $R\ll1$ and thus w.h.p. $\cnp$ contains a $k$-square.
\end{proof}

\subsection{Upper and lower consecutive patterns}\label{sectMonotonePatts}

Although the presence of an exact pattern is not even a convex property,
we can weaken our notion of a pattern in two natural ways to yield properties that are monotone.
The \emph{upper} consecutive pattern
$\gpc[r_1\ldots r_k]$ occurs at position $i$ in a composition $C$ if, for each $j\in[k]$,
we have $C(i-1+j)\geqs r_j$,
and
the \emph{lower} consecutive pattern
$\lpc[r_1\ldots r_k]$ occurs at position $i$ in $C$ if, for each $j\in[k]$,
we have $C(i-1+j)\leqs r_j$.
Thus the presence of an upper pattern is an increasing property, whereas
the presence of a lower pattern is decreasing.

The analysis of upper and lower patterns is similar to that for exact patterns in~Proposition~\ref{propExactPattConsec}.
Indeed, the threshold for the appearance of $\gpc$ is the same as that for $\epc$,
and the threshold for the disappearance of $\lpc$ is the same as that for $\epc$, although in the latter case the threshold probabilities differ.

\begin{prop}\label{propUpperPattConsec}
If $\gpc$ is a nonzero upper consecutive pattern, then for any positive constant~$\alpha$,
\[
  \prob{\text{$\cnp$ contains $\gpc$}} \ssim
  \begin{cases}
    0, & \text{if~ $p\ll n^{-1/|\pi|}$},  \\
    1-e^{-\alpha^{|\pi|}}, & \text{if~ $p\sim\alpha n^{-1/|\pi|}$},  \\
    1, & \text{if~ $n^{-1/|\pi|}\ll p$}.
  \end{cases}
\]
\end{prop}
\begin{proof}
  Suppose $\pi=r_1\ldots r_k$ and $|\pi|=s$.
  For each $i\in[n+1-k]$, let $A_i$ be the event that $\gpc$ occurs at position $i$ in $\cnp$, and let $X$ be the number of occurrences of $\gpc$ in $\cnp$.
  Then,
  $  \prob{A_i}   \eq p^s,  $
  and
  $\expec{X}\sim n p^s$.
  So if $p\ll n^{-1/s}$, then $\expec{X}\ll 1$ and
  w.h.p. $\gpc$ doesn't occur in~$\cnp$.

  Distinct events $A_i$ and $A_j$ are correlated if $t=|j-i|<k$. Then,
  \[
  \prob{A_i\wedge A_j} \eq \prod_{\ell=1}^t p^{r_\ell} \:\times\: \prod_{\ell=1}^{k-t} p^{\max(r_\ell,\,r_{\ell+t})} \:\times \prod_{\ell=k-t+1}^k p^{r_\ell},
  \]
  which is at most $p^{s+1}$, since either $r_\ell>0$ for some $\ell\leqs t$ or else $r_{\ell+t}>r_\ell$ for some $\ell\leqs k-t$.
  Thus,
  \[
  \Delta \defeq \sum_{i\sim j}\prob{A_i\wedge A_j} \sleqs nk p^{s+1}
  \text{~~~~~and~~~~~}
  R \defeq {\Delta}/{\expec{X}^2} \slsim \frac{k}{{np^{s-1}}} .
  \]
   Moreover,
  $
  \Lambda
  \defeq
  \sum_i\prob{A_i}^2 + \sum_{i\sim j} \prob{A_i}\prob{A_j}
  \ssim
  nkp^{2s}
  .
  $

  Suppose $p= \omega n^{-1/s}\ll1$ for some $\omega\gg1$. Then 
  $\expec{X}\sim\omega^s\gg1$
  and $R\lesssim kp/\omega^s\ll1$.
  Hence, 
  w.h.p. $\gpc$ occurs in $\cnp$.

  Suppose $p\sim\alpha n^{-1/s}$.
  Then $\expec{X}\sim \alpha^s$ and $\Delta\leqs \alpha^s kp\ll 1$, and $\Lambda\sim \alpha^{2s}k/n\ll1$.
  So the number of occurrences of $\gpc$ is asymptotically Poisson with mean $\alpha^s$.
\end{proof}

For a lower consecutive pattern $\lpc{}$, the probability at the threshold depends on the number of distinct patterns weakly dominated by $\pi$ of the same length.

\begin{prop}\label{propLowerPattConsec}
If $\lpc[r_1\ldots r_k]$ is a lower consecutive pattern, then for any positive constant~$\alpha$,
\[
  \prob{\text{$\cnp$ contains $\lpc[r_1\ldots r_k]$}} \ssim
  \begin{cases}
    1, & \text{if~ $q\gg n^{-1/k}$},  \\
    1-e^{-\alpha^k\rho}, & \text{if~ $q\sim\alpha n^{-1/k}$},  \\
    0, & \text{if~ $n^{-1/k}\gg q$} ,
  \end{cases}
\]
where $\displaystyle\rho=\prod_{i\in[k]}(r_i+1)$.
\end{prop}
\begin{proof}
For each $i\in[n+1-k]$, let $A_i$ be the event that $\lpc$ occurs at position $i$ in $\cnp$, and let $X$ be the number of occurrences of $\lpc$ in $\cnp$.
  Then, if $q\ll1$,
  \[
  \prob{A_i} \eq \prod\limits_{\ell\in[k]}(1-p^{r_\ell+1}) \eq \prod\limits_{\ell\in[k]}\big(1-(1-q)^{r_\ell+1}\big)
  \ssim \prod\limits_{\ell\in[k]}(r_\ell+1)q
  \eq \rho q^k
  ,
  \]
  and $\expec{X}\sim \rho n q^k$.
  So if $q\ll n^{-1/k}$, then $\expec{X}\ll 1$ and w.h.p. $\lpc$ doesn't occur in~$\cnp$.

  Distinct events $A_i$ and $A_j$ are correlated if $t=|j-i|<k$. Then, if $q\ll1$,
  \[
  \prob{A_i\wedge A_j}
  \eq \prod_{\ell=1}^t (1-p^{r_\ell+1}) \:\times\: \prod_{\ell=1}^{k-t} (1-p^{\min(r_\ell,\,r_{\ell+t})+1}) \:\times \prod_{\ell=k-t+1}^k (1-p^{r_\ell+1})
  \sleqs \rho^2q^{k+1} .
  \]
   Thus,
  \[
  \Delta \defeq \sum_{i\sim j}\prob{A_i\wedge A_j} \sleqs \rho^2nkq^{k+1}
  \text{~~~~~and~~~~~}
  R \defeq {\Delta}/{\expec{X}^2} \slsim \frac{k}{{nq^{k-1}}} .
  \]
   Moreover,
  $
  \Lambda
  \defeq
  \sum_i\prob{A_i}^2 + \sum_{i\sim j} \prob{A_i}\prob{A_j}
  \ssim
  \rho^2nkq^{2k}
  .
  $

  Suppose $q= \omega n^{-1/k}\ll1$ for some $\omega\gg1$.
  Then 
  $\expec{X}\sim\rho\omega^k\gg1$ and
  $R\lesssim kq/\omega^k\ll1$.
  Hence, 
  w.h.p. $\lpc$ occurs in $\cnp$.

  Suppose $q\sim\alpha n^{-1/k}$.
  Then $\expec{X}\sim \alpha^k\rho$ and $\Delta\leqs \alpha^k\rho^2 kq\ll 1$, and $\Lambda\sim \alpha^{2k}\rho^2k/n\ll1$.
  So the number of occurrences of $\lpc$ is asymptotically Poisson with mean $\alpha^k\rho$.
\end{proof}

\subsubsection{The largest term}\label{sectLargestTerm}

We conclude 
our investigation of upper and lower consecutive patterns
by considering the largest and smallest terms in~$\cnp$.
If $\tmax(C)$ is the largest term in a composition~$C$, then the following is an immediate consequence of Proposition~\ref{propUpperPattConsec}.

\begin{prop}\label{propLargestTerm}
If $r$ is a positive integer, then
  \[
  \prob{\tmax(\cnp)\geqs r} \ssim
  \begin{cases}
    0, & \text{if~ $p\ll n^{-1/r}$},  \\
    1-e^{-\alpha^{r}}, & \text{if~ $p\sim\alpha n^{-1/r}$},  \\
    1, & \text{if~ $n^{-1/r}\ll p$}.
  \end{cases}
  \]
\end{prop}
Thus $\tmax(\cnp)=r$ a.a.s. if $n^{-1/r}\ll p\ll n^{-1/(r+1)}$.

However, the analysis of largest term doesn't require Proposition~\ref{propUpperPattConsec}, since we have an exact expression for the probability that all the terms are small: $\prob{\tmax(\cnp)< r}=(1-p^r)^n$.
For small $p$ we have the following:

\begin{prop}\label{propLargestTermSmallP}
  Suppose $\omega\gg1$ and $p=\omega^{-1}$. Then, for any constant $c$,
  \[
  \prob{\tmax(\cnp) \geqs r} \ssim
  \begin{cases}
    0, & \text{if~ $r-\frac{\log n}{\log \omega}\gg\frac1{\log \omega}$},  \\[3pt]
    1-e^{-e^{-c}}, & \text{if~ $r=\frac{\log n + c}{\log \omega}$},  \\[3pt]
    1, & \text{if~ $\frac{\log n}{\log \omega}-r\gg\frac1{\log \omega}$}.
  \end{cases}
  \]
\end{prop}
\begin{proof}\belowdisplayskip=-12pt
  If $r=\frac{\log n + \delta}{\log \omega}$, then
  \[
  \prob{\tmax(\cnp)< r}
  \eq \left(1-\frac{e^{-\delta}}n\right)^{\!n}
  \ssim
  \begin{cases}
  1, & \text{if~ $\delta\gg1$}, \\
  e^{-e^{-c}}, & \text{if~ $\delta\sim c$}, \\
  0, & \text{if~ $-\delta\gg1$}.
  \end{cases}
  \]
\end{proof}
Thus, w.h.p. $\tmax(\cnp)\sim\log n/\log(1/p)$ if $p\ll1$, and a.a.s. the largest term exhibits \emph{two-point concentration}, that is it takes one of only two possible values.

This is not true if $p$ is constant.
In this case, if $L_p = \log_{1/p}n$ then
the probability that the largest term
differs from $L_p$ by a constant is bounded away from both 0 and 1:
\begin{prop}\label{propLargestTermConstantP}
  If $p$ is constant, then for any constant $c$,
  \[
  \prob{\tmax(\cnp) \geqs r} \ssim
  \begin{cases}
  0, & \text{if~ $r-\log_{1/p}n\gg1$}, \\[3pt]
  1-e^{-p^c}, & \text{if~ $r=\log_{1/p}n+c$}, \\[3pt]
  1, & \text{if~ $\log_{1/p}n-r\gg1$}.
  \end{cases}
  \]
\end{prop}
\begin{proof}\belowdisplayskip=-12pt
If $r=\log_{1/p}n+\delta$, then
\[
  \prob{\tmax(\cnp)< r}
  \eq \left(1-\frac{p^\delta}n\right)^{\!n}
  \ssim
  \begin{cases}
  1, & \text{if~ $\delta\gg1$}, \\
  e^{-p^c}, & \text{if~ $\delta\sim c$}, \\
  0, & \text{if~ $-\delta\gg1$}.
  \end{cases}
\]
\end{proof}
Thus, when $p$ is constant, the distribution of the largest term is concentrated around $\log_{1/p}n$.
For a more detailed analysis, including the very slight oscillatory behaviour of its mean
(dependent on the fractional part of $\log_{1/p}n$),
and the determination of its (constant) variance,
see~\cite[Section~4.1]{LPW2005} or the exposition (for $p=\frac12$) in~\cite[Proposition V.1]{FS2009}.

When $p$ tends to 1, we have the following bounds on the largest term:
\begin{prop}\label{propLargestTermSmallQ}
  If $q\ll 1$, then for any $\veps>0$,
  \[
  \prob{\tmax(\cnp) \geqs r} \ssim
  \begin{cases}
  0, & \text{if~ $r\geqs(1+\veps)q^{-1}\log n$}, \\[3pt]
  1, & \text{if~ $r\leqs(1-\veps)q^{-1}\log n$} .
  \end{cases}
  \]
\end{prop}
\begin{proof}
  Let
  $
  L \eq \log \prob{\tmax(\cnp)< r}
  \eq
  n\log\!\big(1-\exp\!\big(r\log(1-q)\big)\big) .
  $

  Now, for small enough $x$, we have $-2x \sless \log(1-x) \sless -x$.

  So, if $r=(1+\veps)q^{-1}\log n$, then for sufficiently large $n$,
  \[
  L 
    \sgtr n\log\!\big(1-\exp(-rq)\big) \\
    \eq   n\log\!\big(1-n^{-(1+\veps)}\big) \\
    \sgtr -2n^{-\veps}.
  \]
  Thus $L\sim 0$, and $\prob{\tmax(\cnp)< r}\sim 1$.

  Similarly,
  using the tighter bound $-x-x^2<\log(1-x)$,
  if now $r=(1-\veps)q^{-1}\log n$, then for sufficiently large $n$,
  \[
  L 
  \sless n\log\!\big(1-\exp(-rq(1+q))\big) \\
    \eq    n\log\!\big(1-n^{-(1-\veps)(1+q)}\big) \\
    \sless -n^{\veps-(1-\veps)q},
  \]
  Thus $\liminfty L= -\infty$, and $\prob{\tmax(\cnp)< r}\sim 0$.
\end{proof}
Hence, when $m\gg n$, a.a.s. $\tmax(\cnm)\sim\frac{m}n\log n$, a factor of $\log n$ more than the value of the average term.

\subsubsection{The smallest term}\label{sectSmallestTerm}

We now turn briefly to consider the smallest term in $\cnp$.
If $\tmin(C)$ is the smallest term in a composition~$C$, then the following is an immediate consequence of Proposition~\ref{propLowerPattConsec}.

\begin{prop}\label{propSmallestTerm}
If $r$ is a positive integer, then for any positive constant~$\alpha$,
\[
  \prob{\tmin(\cnp)\geqs r} \ssim
  \begin{cases}
    0, & \text{if~ $q\gg n^{-1}$},  \\
    e^{-\alpha r}, & \text{if~ $q\sim\alpha n^{-1}$},  \\
    1, & \text{if~ $n^{-1}\gg q$} .
  \end{cases}
\]
\end{prop}
Thus, a.a.s. $\tmin(\cnp)=0$ as long as $q\gg n^{-1}$, but
 $q\asymp n^{-1}$ is the threshold for $\tmin(\cnp)$ to exceed \emph{any} fixed positive value.
This somewhat paradoxical phenomenon is perhaps a little easier to understand if one recalls that the most likely value taken by any term in $\cnp$ is \emph{zero} --- whatever the values of $n$ and $p$.

More generally, we have the following.
\begin{prop}\label{propSmallestTermGrowing}
If $r\gg1$, then for any positive constant~$\alpha$,
\[
  \prob{\tmin(\cnp)\geqs r} \ssim
  \begin{cases}
    0, & \text{if~ $q\gg 1/rn$},  \\
    e^{-\alpha}, & \text{if~ $q\sim\alpha/rn$},  \\
    1, & \text{if~ $1/rn\gg q$} .
  \end{cases}
\]
\end{prop}

\begin{proof}\belowdisplayskip=-12pt
Suppose $\omega\gg1$. Then,
  \[
  \prob{\tmin(\cnp)\geqs r} \eq (1-q)^{rn} \eq
  \begin{cases}
    (1-\omega/rn)^{rn}   \:\lesssim\: e^{-\omega}   \:\ll\:  1 , & \text{if~ $q=\omega/rn$}    , \\[3pt]
    (1-\alpha/rn)^{rn}   \:\sim\:     e^{-\alpha}              , & \text{if~ $q\sim\alpha/rn$} , \\[3pt]
    (1-1/\omega rn)^{rn} \:\sim\:     e^{-1/\omega} \:\sim\: 1 , & \text{if~ $q=1/\omega rn$}  .
  \end{cases}
  \]
\end{proof}
Hence, once $m\gg n^2$, a.a.s. $\tmin(\cnm)\asymp m/n^2$.

\subsection{Consecutive patterns specifying relative ordering}\label{sectOrderingPatts}

A classical notion of pattern containment is that a pattern specifies an {ordering}.
We call such a pattern an \emph{ordering} pattern, to distinguish from other types of pattern in this work, and
represent ordering patterns without a prefix.
The consecutive ordering pattern 
$\opc$ of length $k$
occurs at position $\ell$ in a composition $C$ if the relative order of $C(\ell),\ldots,C(\ell-1+k)$ is the same as that of 
$\pi(1),\ldots,\pi(k)$.
That is, if, for every $i,j\in[k]$,
\begin{align*}
C(\ell-1+i) \:<\: C(\ell-1+j) & \;\Longleftrightarrow\; \pi(i) \:<\: \pi(j) , \\
C(\ell-1+i) \:=\: C(\ell-1+j) & \;\Longleftrightarrow\; \pi(i) \:=\: \pi(j) , \\
C(\ell-1+i) \:>\: C(\ell-1+j) & \;\Longleftrightarrow\; \pi(i) \:>\: \pi(j) .
\end{align*}
For example, $\opc[0211]$ occurs at position $\ell$ in $C$ if $C(\ell)<C(\ell+2)=C(\ell+3)<C(\ell+1)$.
The composition in Figure~\ref{figComposition} on page~\pageref{figComposition} contains a single occurrence of $\opc[0211]$, formed by the consecutive terms 1633.

For uniqueness, we require that the set of numbers used in an ordering pattern is an initial segment of the nonnegative integers.
Thus, $\opc[0211]$ is an ordering pattern, whereas $\opc[0322]$, $\opc[1322]$ and $\opc[4977]$ are not valid patterns.
For results concerning ordering patterns from an enumerative perspective, see the paper by
Kitaev, McAllister and Petersen~\cite{KMcAP2006}
and the papers by
Heubach and Mansour~\cite{HM2005,HM2007}.

Let $P(\opc)$ denote the probability that the ordering pattern $\opc$ occurs at some specified position in~$\cnp$.
Then, since terms in $\cnp$ are independent, $P(\opc)$ does not depend on the order of numbers in~$\pi$.
For example, $P(\opc[0211])=P(\opc[1021])=P(\opc[0112])$.

Before we proceed further, we need a technical result relating the probability of a ``shifted'' occurrence of $\opc$ to the value of the ``unshifted'' probability~$P(\opc)$.

\begin{prop}\label{propShiftedPatt}
  Suppose $E_i^h(\opc)$ is the event that $\opc$ occurs at position $i$ in $\cnp$ with every term of its occurrence at least $h$.
  If $\opc$ 
  has length $k$ and $i\in[n+1-k]$, then
  $\prob{E_i^h(\opc)} \eq p^{hk}P(\opc)$.
\end{prop}
\begin{proof}
  Observe that, by definition,
  \[P(\opc) \eq \sum\limits_{\rho\cong\pi} P(\epc[\rho]),\]
  where the sum is over all the exact patterns that are order-isomorphic to~$\pi$.
  Also,
  \[\{ \rho \::\: \rho\cong\pi \:\wedge\: \tmin(\rho)\geqs h \} \eq \{ \rho+h \::\: \rho\cong\pi \},\]
  where $\rho+h$ is the result of adding $h$ to each term of $\rho$.

  Now
  $P(\epc[\rho+h])\eq p^{hk}P(\epc[\rho]).$
  So,
  \[
  \prob{E_i^h(\opc)}
  \eq \sum\limits_{\rho\cong\pi} P(\epc[\rho+h])
  \eq \sum\limits_{\rho\cong\pi} p^{hk}P(\epc[\rho])
  \eq p^{hk} \sum\limits_{\rho\cong\pi} P(\epc[\rho])
  \eq p^{hk} P(\opc) .
  \qedhere
  \]
\end{proof}

With this result, we can determine the probability of an ordering pattern occurring at some specified position in $\cnp$.

\subsubsection{Consecutive ordering patterns with distinct terms}\label{sectTotalOrderingPatts}

We begin by considering patterns 
in which every term is distinct, such as $\opc[102]$ or $\opc[3120]$.
These patterns are permutations of an initial segment of the nonnegative integers, and
an occurrence of such a pattern induces a total order on the relevant terms, so we call an ordering pattern in which every term is distinct a \emph{total ordering pattern}.

\begin{prop}\label{propTotalOrderProb}
  Suppose $\opc$ is a total ordering pattern of length $k$. Then, for each $i\in[n+1-k]$,
  \[
  \prob{\text{$\opc$ occurs at position $i$ in $\cnp$}}
  \eq
  P(\opc)
  \eq
  \prod_{j=1}^k \frac{q p^{j-1}}{1-p^j}
  \eq
  p^{k(k-1)/2} \, \prod_{j=1}^k \frac{q}{1-p^j} .
  \]
\end{prop}
\begin{proof}
  Without loss of generality, let $\pi=01\ldots(k-1)$.
  We proceed by induction on the length of the pattern.
  Trivially, $P(\opc[0])=1$.

  Suppose $k>1$. Let $\pi_1=01\ldots(k-2)$, and assume that $P(\opc[\pi_1])$ satisfies the statement of the proposition.
  Then, by Proposition~\ref{propShiftedPatt},
  \[
  P(\opc)
  \eq \sum_{h=0}^\infty \prob{(\cnp(i)=h) \:\wedge\: E_{i+1}^{h+1}(\opc[\pi_1])}
  \eq \sum_{h=0}^\infty qp^h p^{(h+1)(k-1)} P(\opc[\pi_1])
  \eq \frac{q p^{k-1}}{1-p^k} P(\opc[\pi_1])
  ,
  \]
  as required.
\end{proof}

If $\opc$ is a total ordering pattern of length $k\geqs2$, then
by Proposition~\ref{propExactPattConsec}, 
$p\asymp n^{-1/|\pi|}=n^{-2/k(k-1)}$ is the threshold in $\cnp$ for the arrival of the exact pattern $\epc$, with any other exact pattern order-isomorphic to $\pi$ arriving later.
Thus $p\asymp n^{-2/k(k-1)}$ is also the threshold in $\cnp$ for the appearance of~$\opc$.
However, unlike the situation with exact patterns, a.a.s. a total ordering pattern $\opc$ does not disappear as $p$ tends to~1.
Indeed, once $q\ll1$, the relative ordering of $k$ consecutive terms of $\cnp$ is asymptotically uniformly distributed over the $k!$ permutations of length~$k$.

\begin{prop}\label{propTotalOrderProbLimit}
  If $\opc$ is a total ordering pattern of length $k$, and $q\ll1$,
  then, for each $i\in[n+1-k]$,
  \[
  \prob{\text{$\opc$ occurs at position $i$ in $\cnp$}}
  \ssim
  \frac1{k!}.
  \]
\end{prop}
\begin{proof}
\[
  P(\opc)
  \eq
  \prod_{j=1}^k \frac{q p^{j-1}}{1-p^j}
  \ssim
  \prod_{j=1}^k \frac1{1+p+\ldots+p^{j-1}}
  \ssim
  \prod_{j=1}^k \frac1j
  \eq
  \frac1{k!} .
  \qedhere
\]
\end{proof}

A natural question in the context of total ordering patterns is to determine
the length of the longest increasing run in $\cnp$.
If $p\ll n^{-\delta}$ for some $\delta>0$, then Proposition~\ref{propExactPattConsec} together with Proposition~\ref{propTotalOrderProbLimit} gives us the following:
\begin{prop}\label{propTotalOrderingPattConsec}
  If $\opc$ is a total ordering pattern of length $k\geqs 2$, then
  \[
  \prob{\text{$\cnp$ contains $\opc$}}
  \ssim
  \begin{cases}
    0, & \text{if~ $p\ll n^{-2/k(k-1)}$},  \\
    1-e^{-\alpha^{k(k-1)/2}}, & \text{if~ $p\sim\alpha n^{-2/k(k-1)}$},  \\
    1, & \text{if~ $n^{-2/k(k-1)}\ll p$}.
  \end{cases}
  \]
\end{prop}

Once $p$ grows faster than any negative exponential function, then the asymptotic length of the longest increasing run is given by the following proposition:

\begin{prop}\label{propIncreasingRun2}
Suppose $n^{-\delta}\ll p\ll 1$ for every positive $\delta$. Then, for any $\veps>0$,
  \[
  \prob{\text{$\cnp$ contains an increasing run of length $k$}} \ssim
  \begin{cases}
    0, & \text{if~ $k\geqs(1+\veps)\sqrt{2\log_{1/p} n}$},  \\[6pt]
    1, & \text{if~ $k\leqs(1-\veps)\sqrt{2\log_{1/p} n}$}.
  \end{cases}
  \]
\end{prop}
\begin{proof}
  For each $i\in[n+1-k]$, let $A_i$ be the event that $\opc[01\ldots(k-1)]$ occurs at position $i$ in $\cnp$, and let $X$ be the number of increasing runs of length $k$ in $\cnp$.

  Then,
  $
  \prob{A_i}
  \ssim q^k p^{k(k-1)/2} ,
  $
  and so
  $\expec{X}\sim n q^k p^{k(k-1)/2}$
  if $k\ll n$.

  Suppose $\alpha>1$ and $k=\alpha \sqrt{2\log_{1/p} n}$. Since $\expec{X}\lesssim n p^{k(k-1)/2}$, we have
  \begin{align*}
  \log \expec{X} & \slsim \log n 
                    \:+\: \alpha^2 \log p \,\log_{1/p}n \:-\: \alpha\log p\, \sqrt{2\log_{1/p}n} \\
                 & \eq    -(\alpha^2-1)\log n 
                    \:+\: \alpha \sqrt{2\log (1/p) \log n} .
  \end{align*}
  Now $\log(1/p)\ll\log n$, since $p\gg n^{-\delta}$ for every $\delta>0$.
  Thus, $\liminfty \log\expec{X} = -\infty$. Hence, $\expec{X}\ll 1$ and w.h.p. there are no increasing runs of length $k$ in~$\cnp$.

  Now suppose $0<\alpha<1$ and $k=\alpha \sqrt{2\log_{1/p} n}$. Since $\expec{X}\gtrsim n q^k p^{k^2/2}$, we have
  \begin{align*}
  \log \expec{X} & \sgsim \log n \:+\: \alpha\log(1-p)\, \sqrt{2\log_{1/p}n} \:+\: \alpha^2 \log p \,\log_{1/p}n \\
                 & \sgsim  
                 (1-\alpha^2)\log n \:-\: 2\alpha p\sqrt{2\log_{1/p}n}
                 \sgg 1.
  \end{align*}

  Distinct events $A_i$ and $A_j$ are correlated if $t=|j-i|<k$. Then,
  \[
  \prob{A_i\wedge A_j} \ssim q^{k+t} p^{(k+t)(k+t-1)/2} \sless q^{k+1} p^{k(k+1)/2}.
  \]
  So $\Delta \defeq \sum\limits_{i\sim j}\prob{A_i\wedge A_j} \slsim nk q^{k+1} p^{k(k+1)/2}$, and
  \[
  R \defeq {\Delta}/{\expec{X}^2} \slsim \frac{k}{n q^{k-1} p^{k(k-3)/2}} \sless \frac{k}{n q^k p^{k^2/2}}.
  \]
  Thus,
  \begin{align*}
  \log R
  & \slsim  \log k \:-\: \log n \:-\: \alpha \log(1-p)\,\sqrt{2\log_{1/p} n} \:-\:  \alpha^2 \log p \,\log_{1/p} n \\
  & \slsim  \log k \:-\: (1-\alpha^2)\log n \:+\: 2 \alpha p \sqrt{2\log_{1/p} n}.
  \end{align*}
  So $\liminfty \log R = -\infty$.
  Hence, $R \ll 1$ and w.h.p. $\cnp$ contains an increasing run of length~$k$.
\end{proof}

Thus, for example, when $p\sim1/\log n$, a.a.s. the longest increasing run in $\cnp$ has length close to $\sqrt{2\log n/\log\log n}$.
For the length of the longest increasing run when $p$ is constant or $q\ll1$, in which case the contribution from $\prod_{j=1}^k(1-p^j) = (p;p)_k$ in the denominator of $P(\opc[01\ldots(k-1)])$ is nontrivial, see the work of Louchard and Prodinger~\cite{LP2003a}.

\subsubsection{Consecutive ordering patterns with repeated terms}\label{sectNonTotalOrderingPatts}

Let us now move on to ordering patterns, such as $\opc[12201]$, in which at least one term is repeated.
Since $P(\opc)$ is independent of the order of the terms in $\pi$,
when calculating $P(\opc)$ we may assume that the terms of $\pi$ are weakly increasing
and that $\pi$ has the form $0^{\ell_0}1^{\ell_1}\ldots r^{\ell_r}$, in which the exponents record the number of occurrences of each distinct term (for example, $0^31^12^2=000122$).

\begin{prop}\label{propOrderPattProb}
  Suppose $\pi=0^{\ell_0}1^{\ell_1}\ldots r^{\ell_r}$. 
  If $\pi$ has length $k$, then, for each $i\in[n+1-k]$,
  \[
  \prob{\text{$\opc$ occurs at position $i$ in $\cnp$}}
  \eq
  P(\opc)
  \eq
  \prod_{j=0}^r \frac{q^{\ell_j} p^{s_{j+1}}}{1-p^{s_j}}
  \eq
  q^{k-(r+1)} \, p^{|\pi|} \, \prod_{j=0}^r \frac{q}{1-p^{s_j}}
  ,
  \]
  where $\displaystyle s_j=\sum\limits_{i=j}^r \ell_i$ for each $j=0,\ldots,r+1$ \textup(with $s_{r+1}=0$\textup).
\end{prop}
\begin{proof}
  The proof is by induction on $r$, and is very similar to that for Proposition~\ref{propTotalOrderProb}.
  If $r=0$, then $s_0=\ell_0=k$ and
  \[
  P(\opc) \eq
  \sum_{h=0}^\infty (qp^h)^k \eq \frac{q^{\ell_0}}{1-p^{s_0}} .
  \]
  Suppose now that $r>1$. Let $\pi_1=0^{\ell_1}1^{\ell_2}\ldots (r-1)^{\ell_r}$, and assume that $P(\opc[\pi_1])$ satisfies the statement of the proposition.
  Then, using Proposition~\ref{propShiftedPatt},
  \begin{align*}
  P(\opc)
  &\eq \sum_{h=0}^\infty \prob{\text{$\epc[h^{\ell_0}]$ occurs at position $i$ in $\cnp$ ~and~ $E_{i+\ell_0}^{h+1}(\opc[\pi_1])$ } } \\
  &\eq \sum_{h=0}^\infty (qp^h)^{\ell_0}\, p^{(h+1)s_1} P(\opc[\pi_1])
  \eq \frac{q^{\ell_0} p^{s_1}}{1-p^{s_0}} P(\opc[\pi_1])
  ,
  \end{align*}
  as required.
\end{proof}

If $\pi$ is nonzero, then
by Proposition~\ref{propExactPattConsec}, 
$p\asymp n^{-1/|\pi|}$ is the threshold in $\cnp$ for the arrival of the exact pattern $\epc$, with any other exact pattern order-isomorphic to $\pi$ arriving later.
Thus $p\asymp n^{-1/|\pi|}$ is also the threshold in $\cnp$ for the appearance of the ordering pattern $\opc$.
However, unlike in the case of total ordering patterns, an ordering pattern with a repeated term also exhibits a threshold for its disappearance:

\begin{prop}\label{propOrderPattThreshold}
  Suppose the multiset of terms in $\pi$ is $\{0^{\ell_0},1^{\ell_1},\ldots,r^{\ell_r}\}$, where $\ell_j>1$ for at least one term~$j$.
  If $\pi$ has length $k$ and $d=k-(r+1)$, then
for any positive constant~$\alpha$,
\[
  \prob{\text{$\cnp$ contains $\opc$}} \ssim
  \begin{cases}
    1, & \text{if~ $1\gg q\gg n^{-1/d}$},  \\
    1-e^{-\alpha^d/\lambda}, & \text{if~ $q\sim\alpha n^{-1/d}$},  \\
    0, & \text{if~ $n^{-1/d}\gg q$} ,
  \end{cases}
\]
where $\displaystyle\lambda = \prod\limits_{j=0}^r \,\sum\limits_{i=j}^r \ell_i$.
\end{prop}
\begin{proof}
For each $i\in[n+1-k]$, let $A_i$ be the event that $\opc$ occurs at position $i$ in $\cnp$, and let $X$ be the number of occurrences of $\opc$ in $\cnp$.
If $q\ll1$, then
\[
    \prob{A_i}
    \ssim q^d  \, \prod_{j=0}^r \frac{1-p}{1-p^{s_j}}
    \eq   q^d  \, \prod_{j=0}^r \frac1{1+p+\ldots+p^{s_j-1}}
    \ssim q^d  \, \prod_{j=0}^r \frac1{s_j}
    \eq   \frac{q^d}{\lambda}
    ,
\]
where $\displaystyle s_j=\sum\limits_{i=j}^r \ell_i$ for each $j=0,\ldots,r+1$.
So $\expec{X}\sim n q^d/\lambda$.

If $q\ll n^{-1/d}$, then $\expec{X}\ll 1$ and w.h.p. $\opc$ doesn't occur in~$\cnp$.

Distinct events $A_i$ and $A_j$ are correlated if $t=|j-i|<k$.
There may be several ways in which $A_i$ and $A_j$ may consistently overlap, corresponding to a variety of longer ordering patterns.
For example, if $\pi=010$, then there are three possible arrangements: $01010$, $01020$ and $02010$ (each with $t=2$).
However, for each~$t$, there are certainly never more than $(k+t)^t<(2k)^k$ possibilities.

Thus,
 \[
 \Delta \defeq \sum_{i\sim j}\prob{A_i\wedge A_j} \sleqs \kappa n q^{d+1}
 \text{~~~~~and~~~~~}
 R \defeq {\Delta}/{\expec{X}^2} \slsim \frac{\kappa\lambda^2}{{nq^{d-1}}} ,
 \]
 for some constant $\kappa$, dependent only on $\pi$.

 Moreover,
  $
  \Lambda
  \defeq
  \sum_i\prob{A_i}^2
  +
  \sum_{i\sim j} \prob{A_i}\prob{A_j}
  \ssim
  nkq^{2d}/\lambda^2
  .
  $

  Suppose $q= \omega n^{-1/d}\ll1$ for some $\omega\gg1$, then $\expec{X}\sim\omega^d/\lambda\gg1$ and
  $R\lesssim \kappa\lambda^2q/\omega^d\ll1$. So w.h.p. $\opc$ occurs in $\cnp$.

  Finally, suppose $q=\alpha n^{-1/d}$.
  Then $\expec{X}\sim \alpha^d/\lambda$ and $\Delta\leqs \alpha^d \kappa q\ll 1$, and $\Lambda\sim \alpha^{2d}k/\lambda^2n\ll1$, so
  the number of occurrences of $\opc$ is asymptotically Poisson with mean~$\alpha^d/\lambda$.
\end{proof}

Thus, for example,
the threshold for the disappearance from $\cnp$ of \emph{balanced peaks} $\opc[010]$ is at $q\asymp n^{-1}$,
as is the threshold for the disappearance of \emph{balanced valleys} $\opc[101]$.
However, at this threshold,
$\cnp$ is more likely to contain a balanced valley 
than it is to contain a balanced peak. 

A composition in which no pair of adjacent terms are equal (that is, avoiding the ordering pattern~$\opc[00]$) is called a \emph{Carlitz composition}.
These have been well-studied~\cite{GH2002,Kheyfets2005,KP1998,LP2002}.
More generally, the distribution of the lengths of runs of consecutive equal terms in compositions has been investigated~\cite{GKP2003,Wilf2011}.
From our evolutionary perspective, we have the following threshold for the disappearance of runs of equal terms (a direct consequence of Proposition~\ref{propOrderPattThreshold}):

\begin{prop}\label{propCarlitzThreshold} For any fixed $k\geqs2$ and positive constant $\alpha$,
\[
  \prob{\text{$\cnp$ contains a run of $k$ equal terms}} \ssim
  \begin{cases}
    1, & \text{if~ $1\gg q\gg n^{-1/(k-1)}$},  \\
    1-e^{-\alpha^{k-1}/k}, & \text{if~ $q\sim\alpha n^{-1/(k-1)}$},  \\
    0, & \text{if~ $n^{-1/(k-1)}\gg q$} .
  \end{cases}
\]
\end{prop}

Comparison with Proposition~\ref{propEqualRuns1} shows that this threshold is the same as for the disappearance of runs of equal \emph{nonzero} terms.
In particular, the threshold for $\cnp$ to become a Carlitz composition is at $q\asymp n^{-1}$.

\subsection{Nonconsecutive patterns}\label{sectNonConsecPatts}

In this concluding section we finally remove the restriction that the occurrence of a pattern must be consecutive.
Nonconsecutive patterns are represented without an overline.
We begin with exact patterns.
The pattern
$\ep[r_1\ldots r_k]$ occurs in a composition $C$ if there exists a sequence of indices $i_1<i_2<\ldots<i_k$ such that for each $j\in[k]$,
we have $C(i_j)=r_j$.
The threshold for the appearance of a nonconsecutive exact pattern depends only on its maximum term.
In contrast, all such patterns share the same threshold for their disappearance.

\begin{prop}\label{propExactPatt}
If $\ep$ is a nonzero exact pattern whose largest term is $r$, then
\[
  \prob{\text{$\cnp$ contains $\ep$}} \ssim
  \begin{cases}
    0, & \text{if~ $p\ll n^{-1/r}$},  \\
    1, & \text{if~ $n^{-1/r}\ll p$ ~and~ $q\gg n^{-1}$},  \\
    0, & \text{if~ $n^{-1}\gg q$} .
  \end{cases}
\]
\end{prop}
\begin{proof}
  If $p\ll n^{-1/r}$ then, by Proposition~\ref{propExactPattConsec}, w.h.p. $\cnp$ has no term equal to $r$ and thus contains no occurrence of~$\ep$.
  Also, by Proposition~\ref{propSmallestTerm}, if $q\ll n^{-1}$ then w.h.p. the smallest term of $\cnp$ exceeds any fixed value, so there is no occurrence of~$\ep$.

  Suppose $\pi$ has length $k$. For each $j\in[k]$, let $i_j=\floor{jn/k}$ and $\C_j=\big(\cnp(i_{j-1}+1),\ldots,\cnp(i_j)\big)$,
  so that $\C_1,\ldots,\C_k$ is a partition of the terms of $\cnp$.
  If $n^{-1/r}\ll p$ and $q\gg n^{-1}$, then, for each $j\in[k]$, by Proposition~\ref{propExactPattConsec}, w.h.p. $\C_j$ has a term equal to $\pi(j)$, so $\cnp$ contains an occurrence of~$\ep$.
\end{proof}

Nonconsecutive upper and lower patterns exhibit analogous behaviour. The proofs are very similar to that for Proposition~\ref{propExactPatt}, so are omitted.

\begin{prop}\label{propUpperPatt}
If $\gp$ is a nonzero upper pattern whose largest term is $r$, then
\[
  \prob{\text{$\cnp$ contains $\gp$}} \ssim
  \begin{cases}
    0, & \text{if~ $p\ll n^{-1/r}$},  \\
    1, & \text{if~ $n^{-1/r}\ll p$}.
  \end{cases}
\]
\end{prop}

\begin{prop}\label{propLowerPatt}
If $\lp$ is a lower pattern, then
\[
  \prob{\text{$\cnp$ contains $\lp$}} \ssim
  \begin{cases}
    1, & \text{if~ $q\gg n^{-1}$},  \\
    0, & \text{if~ $n^{-1}\gg q$}.
  \end{cases}
\]
\end{prop}

We can generalise these results to exact \emph{vincular} patterns, in which there are some adjacency requirements, represented by an overline (or \emph{vinculum}).
For example, $C$ contains the exact vincular pattern $\epc[12]3$ if there exist indices $i$ and $j>i+1$ such that $C(i)=1$, $C(i+1)=2$ and $C(j)=3$.
We consider a vincular pattern to consist of a number of consecutive \emph{blocks}, the terms of which must be adjacent in any occurrence.
For example, $\epc[12]\:\!0\:\!4\:\!\opc[003]$ consists of four blocks, of lengths 2, 1, 1 and 3, and sizes 3, 0, 4 and 3, respectively.

The threshold for the appearance of an exact vincular pattern depends on the \emph{size} of its \emph{largest} block, whereas
the threshold for its disappearance depends on the \emph{length} of its \emph{longest} block.
This is analogous to 
the situation for exact consecutive patterns (Proposition~\ref{propExactPattConsec}).
We omit the proof since it is exactly analogous to that for Proposition~\ref{propExactPatt} (with one part in the partition of~$\cnp$ for each block in the pattern).

\begin{prop}\label{propVincularPatt}
If $\ep$ is an exact vincular pattern whose largest block has size $s$ and whose longest block has length $\ell$, then
\[
  \prob{\text{$\cnp$ contains $\ep$}} \ssim
  \begin{cases}
    0, & \text{if~ $p\ll n^{-1/s}$},  \\
    1, & \text{if~ $n^{-1/s}\ll p$ ~and~ $q\gg n^{-1/\ell}$},  \\
    0, & \text{if~ $n^{-1/\ell}\gg q$} .
  \end{cases}
\]
\end{prop}



Let's now consider nonconsecutive total ordering patterns (ordering patterns with distinct terms, such as $102$ or~$3120$).
The threshold for their appearance depends only on their length (or equivalently on their largest term, which is one less than their length).
In the case of constant $p$, the distribution of the number of inversions (occurrences of $10$) has previously been investigated by Prodinger~\cite{Prodinger2001} and by Heubach, Knopfmacher, Mays and Munagi~\cite{HKMM2011}.
\begin{prop}\label{propTotalOrderingPatt}
If $\pi$ is a total ordering pattern of length $k$, then
\[
  \prob{\text{$\cnp$ contains $\pi$}} \ssim
  \begin{cases}
    0, & \text{if~ $p\ll n^{-1/(k-1)}$},  \\
    1, & \text{if~ $n^{-1/(k-1)}\ll p$}.
  \end{cases}
\]
\end{prop}
\begin{proof}
  By Proposition~\ref{propExactPatt}, $p\asymp n^{-1/(k-1)}$ is the threshold for the appearance of the exact pattern~$\ep$, which doesn't disappear until $q\asymp n^{-1}$.
  And by Proposition~\ref{propTotalOrderingPattConsec}, the consecutive total ordering pattern $\opc$ a.a.s. appears when $p\asymp n^{-2/k(k-1)}$, and doesn't disappear.
  The result follows because an occurrence of $\ep$ is an occurrence of $\pi$, as is an occurrence of~$\opc$.
\end{proof}

Our last result concerns equal terms in $\cnp$.
In particular, Proposition~\ref{propEqualTerms} gives a threshold at $q\asymp n^{-2}$ for the disappearance of (nonconsecutive) ordering pattern~$00$.
Thus, when $q\ll n^{-2}$ a.a.s. no two terms of $\cnp$ are equal, and so nonconsecutive total ordering patterns satisfy the same asymptotic distribution in $\cnp$ as they do in a random permutation.
Hence, when $q\ll n^{-2}$ the distribution of the length of the longest increasing subsequence in $\cnp$ is as described in the
celebrated work of
Baik, Deift and Johansson~\cite{BDJ1999} (see Romik~\cite{Romik2015} for an extended expository presentation).

\begin{prop}\label{propEqualTerms} For any fixed $k\geqs2$ and positive constant $\alpha$,
\[
  \prob{\text{$\cnp$ has at least $k$ equal terms}} \ssim
  \begin{cases}
    1, & \text{if~ $q\gg n^{-k/(k-1)}$},  \\
    1-e^{-\alpha^{k-1}/(k\times k!)}, & \text{if~ $q\sim\alpha n^{-k/(k-1)}$},  \\
    0, & \text{if~ $n^{-k/(k-1)}\gg q$} .
  \end{cases}
\]
\end{prop}
\begin{proof}
  By Proposition~\ref{propLongestComp1}, if $q\gg n^{-1/k}$ then $\cnp$ contains a gap of length at least $k$ (an occurrence of $\epc[0^k]$) and thus has at least $k$ equal (zero) terms. Suppose now that $q\ll 1$.

  Given a vector $\mathbf{i} := (i_1,i_2,\ldots,i_k)\in [n]^k$ such that $i_{j+1}\geqs i_j$ for each $j\in[k-1]$, let $A_\mathbf{i}$ be the event that $\cnp(i_1)=\cnp(i_2)=\ldots=\cnp(i_k)$. Then
  \[
  \prob{A_\mathbf{i}} \eq \sum_{r=0}^\infty q^k p^{kr} \eq \frac{q^k}{1-p^k} \ssim \frac{q^{k-1}}k \qquad (q\ll1) .
  \]
  If $X$ is the total number of $k$-tuples of equal-valued terms, then
  \[
  \expec{X} \eq \binom{n}k\prob{A_\mathbf{i}} \ssim \frac{n^k q^{k-1}}{k\times k!} .
  \]
  So if $q\ll n^{-k/(k-1)}$, then $\expec{X}\ll1$ and w.h.p. there is no occurrence of $0^k$ in $\cnp$.

  Distinct events $A_\mathbf{i}$ and $A_\mathbf{j}$ are correlated ($\mathbf{i}\sim \mathbf{j}$) if the indices in $\mathbf{i}$ and $\mathbf{j}$ have nonempty intersection. Thus,
  \[
            \Delta
  \defeq    \sum_{\mathbf{i}\sim \mathbf{j}} \prob{A_\mathbf{i}\wedge A_\mathbf{j}}
  \ssim     \sum_{t=k+1}^{2k-1} \binom{n}{t} \binom{t}{k} \frac{q^{t-1}}{t}
  \slsim    \sum_{t=k+1}^{2k-1} n^tq^{t-1} .
  \]
  Hence,
  \[
  R \defeq \Delta/\expec{X}^2 \slsim \sum_{t=1}^{k-1} n^{-t}\, q^{-(t-1)} .
  \]
  Suppose $q\eq \omega n^{-k/(k-1)}$ for some $\omega\gg1$.
  Then $\expec{X}\gg1$ and
  \[
  R \slsim \sum_{t=1}^{k-1} n^{-\tfrac{t-k}{t-1}} \,\omega^{-(t-1)} \sll 1 .
  \]
  Hence, a.a.s. there is an occurrence of $0^k$ in $\cnp$.

  Finally,
  \[
  \Lambda
  \defeq
  \sum_\mathbf{i}\prob{A_\mathbf{i}}^2
  +
  \sum_{\mathbf{i}\sim \mathbf{j}} \prob{A_\mathbf{i}}\prob{A_\mathbf{j}}
  \ssim
  \sum_{t=k}^{2k-1} \binom{n}{t} \binom{t}{k} \left(\frac{q^{k-1}}{k}\right)^{\!2}
  \slsim
  q^{2k-2} \sum_{t=k}^{2k-1} n^t
  .
  \]
  Suppose $q\sim \alpha n^{-k/(k-1)}$. Then $\expec{X}\sim\alpha^{k-1}/(k\times k!)$, and
  \[
  \Delta \slsim \sum\limits_{t=1}^{k-1} n^{-t/(k-1)} \sll 1
  \qquad \text{and} \qquad
  \Lambda \slsim \sum\limits_{t=1}^{k-1} n^{-t} \sll 1
  .
  \]
  Thus the number of occurrences of $0^k$ is asymptotically Poisson with mean $\alpha^{k-1}/(k\times k!)$.
\end{proof}

\HIDE
{
\rulebreak
\begin{center}
  {\Huge\emph{Soli Deo gloria!}}
\end{center}
\rulebreak
}

\bibliographystyle{plain}
{\footnotesize\bibliography{../bib/mybib}}

\begin{thebibliography}{10}

\bibitem{AP2013}
H{\"u}seyin Acan and Boris Pittel.
\newblock On the connected components of a random permutation graph with a
  given number of edges.
\newblock {\em J.~Combin. Theory Ser.~A}, 120(8):1947--1975, 2013.

\bibitem{AS2016}
Noga Alon and Joel~H. Spencer.
\newblock {\em The Probabilistic Method}.
\newblock John Wiley \& Sons, fourth edition, 2016.

\bibitem{AGG1989}
R.~Arratia, L.~Goldstein, and L.~Gordon.
\newblock Two moments suffice for {P}oisson approximations: the {C}hen-{S}tein
  method.
\newblock {\em Ann. Probab.}, 17(1):9--25, 1989.

\bibitem{BDJ1999}
Jinho Baik, Percy Deift, and Kurt Johansson.
\newblock On the distribution of the length of the longest increasing
  subsequence of random permutations.
\newblock {\em J.~Amer. Math. Soc.}, 12(4):1119--1178, 1999.

\bibitem{BE1984}
A.~D. Barbour and G.~K. Eagleson.
\newblock Poisson convergence for dissociated statistics.
\newblock {\em J.~Roy. Statist. Soc. Ser.~B}, 46(3):397--402, 1984.

\bibitem{BE1987}
A.~D. Barbour and G.~K. Eagleson.
\newblock An improved {P}oisson limit theorem for sums of dissociated random
  variables.
\newblock {\em J.~Appl. Probab.}, 24(3):586--599, 1987.

\bibitem{BTPermutations}
David Bevan and Dan Threlfall.
\newblock Thresholds for patterns in random permutations with a given number of
  inversions.
\newblock {\em Electron. J. Combin.}, 31(4):~P4.6, 2024.

\bibitem{BT1987}
B.~Bollob\'as and A.~Thomason.
\newblock Threshold functions.
\newblock {\em Combinatorica}, 7(1):35--38, 1987.

\bibitem{Bollobas2001}
B{\'e}la Bollob{\'a}s.
\newblock {\em Random Graphs}.
\newblock Cambridge University Press, second edition, 2001.

\bibitem{CZL2012}
Ai~Lian Chen, Fu~Ji Zhang, and Hao Li.
\newblock The degree distribution of the random multigraphs.
\newblock {\em Acta Math. Sin. (Engl. Ser.)}, 28(5):941--956, 2012.

\bibitem{Chen1975}
Louis H.~Y. Chen.
\newblock Poisson approximation for dependent trials.
\newblock {\em Ann. Probability}, 3(3):534--545, 1975.

\bibitem{ER1977}
P.~Erd\H{o}s and P.~R\'{e}v\'{e}sz.
\newblock On the length of the longest head-run.
\newblock In {\em Topics in information theory ({S}econd {C}olloq.,
  {K}eszthely, 1975)}, Colloq. Math. Soc. J\'{a}nos Bolyai, Vol. 16, pages
  219--228. North-Holland, 1977.

\bibitem{ER1959}
P.~Erd{\H{o}}s and A.~R{\'e}nyi.
\newblock On random graphs {I}.
\newblock {\em Publ. Math. Debrecen}, 6:290--297, 1959.

\bibitem{ER1960}
P.~Erd{\H{o}}s and A.~R{\'e}nyi.
\newblock On the evolution of random graphs.
\newblock {\em Magyar Tud. Akad. Mat. Kutat\'o Int. K\H{o}zl.}, 5:17--61, 1960.

\bibitem{Eryilmaz2006}
Serkan Eryilmaz.
\newblock A note on runs of geometrically distributed random variables.
\newblock {\em Discrete Math.}, 306(15):1765--1770, 2006.

\bibitem{FS2009}
Philippe Flajolet and Robert Sedgewick.
\newblock {\em Analytic Combinatorics}.
\newblock Cambridge University Press, 2009.

\bibitem{FK1996}
Ehud Friedgut and Gil Kalai.
\newblock Every monotone graph property has a sharp threshold.
\newblock {\em Proc. Amer. Math. Soc.}, 124(10):2993--3002, 1996.

\bibitem{FK2015}
Alan Frieze and Micha{\l} Karo\'nski.
\newblock {\em Introduction to Random Graphs}.
\newblock Cambridge University Press, 2015.

\bibitem{Gafni2015}
Ayla Gafni.
\newblock Longest run of equal parts in a random integer composition.
\newblock {\em Discrete Math.}, 338(2):236--247, 2015.

\bibitem{Gilbert1959}
E.~N. Gilbert.
\newblock Random graphs.
\newblock {\em Ann. Math. Statist.}, 30(4):1141--1144, 1959.

\bibitem{GH2002}
William M.~Y. Goh and Pawe\l Hitczenko.
\newblock Average number of distinct part sizes in a random {C}arlitz
  composition.
\newblock {\em European J. Combin.}, 23(6):647--657, 2002.

\bibitem{GSW1986}
Louis Gordon, Mark~F. Schilling, and Michael~S. Waterman.
\newblock An extreme value theory for long head runs.
\newblock {\em Probab. Theory Relat. Fields}, 72(2):279--287, 1986.

\bibitem{GKP2003}
Peter~J. Grabner, Arnold Knopfmacher, and Helmut Prodinger.
\newblock Combinatorics of geometrically distributed random variables: run
  statistics.
\newblock {\em Theoret. Comput. Sci.}, 297(1--3):261--270, 2003.

\bibitem{GO1980}
L.~J. Guibas and A.~M. Odlyzko.
\newblock Long repetitive patterns in random sequences.
\newblock {\em Z. Wahrsch. Verw. Gebiete}, 53(3):241--262, 1980.

\bibitem{HKMM2011}
S.~Heubach, A.~Knopfmacher, M.~E. Mays, and A.~Munagi.
\newblock Inversions in compositions of integers.
\newblock {\em Quaest. Math.}, 34(2):187--202, 2011.

\bibitem{HM2005}
Silvia Heubach and Toufik Mansour.
\newblock Counting rises, levels, and drops in compositions.
\newblock {\em Integers}, 5(1):~A11, 2005.

\bibitem{HM2007}
Silvia Heubach and Toufik Mansour.
\newblock Enumeration of 3-letter patterns in compositions.
\newblock In B.~Landman, M.~Nathanson, J.~Ne\v{s}et\v{r}il, R.~Nowakowski, and
  C.~Pomerance, editors, {\em Combinatorial number theory}, pages 243--264. de
  Gruyter, 2007.

\bibitem{HM2010}
Silvia Heubach and Toufik Mansour.
\newblock {\em Combinatorics of compositions and words}.
\newblock CRC Press, 2010.

\bibitem{Holst1979}
Lars Holst.
\newblock A unified approach to limit theorems for urn models.
\newblock {\em J.~Appl. Probab.}, 16(1):154--162, 1979.

\bibitem{Huillet2011}
Thierry~E Huillet.
\newblock A {B}ose--{E}instein approach to the random partitioning of an
  integer.
\newblock {\em J.~Statistical Mechanics: Theory and Experiment},
  2011(08):~P08021, 2011.

\bibitem{Janson1994}
Svante Janson.
\newblock Coupling and {P}oisson approximation.
\newblock {\em Acta Appl. Math.}, 34(1-2):7--15, 1994.

\bibitem{Janson2012}
Svante Janson.
\newblock Simply generated trees, conditioned {G}alton-{W}atson trees, random
  allocations and condensation.
\newblock {\em Probab. Surv.}, 9:103--252, 2012.

\bibitem{JLR2000}
Svante Janson, Tomasz {\L}uczak, and Andrzej Ruci\'nski.
\newblock {\em Random Graphs}.
\newblock Wiley, 2000.

\bibitem{Kheyfets2005}
Boris~L. Kheyfets.
\newblock The number of part sizes of a given multiplicity in a random
  {C}arlitz composition.
\newblock {\em Adv. in Appl. Math.}, 35(3):335--354, 2005.

\bibitem{Kitaev2011}
Sergey Kitaev.
\newblock {\em Patterns in Permutations and Words}.
\newblock Springer, 2011.

\bibitem{KMcAP2006}
Sergey Kitaev, Tyrrell~B. McAllister, and T.~Kyle Petersen.
\newblock Enumerating segmented patterns in compositions and encoding by
  restricted permutations.
\newblock {\em Integers}, 6:~A34, 2006.

\bibitem{KP1998}
Arnold Knopfmacher and Helmut Prodinger.
\newblock On {C}arlitz compositions.
\newblock {\em European J. Combin.}, 19(5):579--589, 1998.

\bibitem{Louchard2002}
Guy Louchard.
\newblock Runs of geometrically distributed random variables: a probabilistic
  analysis.
\newblock {\em J.~Comput. Appl. Math.}, 142(1):137--153, 2002.

\bibitem{LP2002}
Guy Louchard and Helmut Prodinger.
\newblock Probabilistic analysis of {C}arlitz compositions.
\newblock {\em Discrete Math. Theor. Comput. Sci.}, 5(1):71--95, 2002.

\bibitem{LP2003a}
Guy Louchard and Helmut Prodinger.
\newblock Ascending runs of sequences of geometrically distributed random
  variables: a probabilistic analysis.
\newblock {\em Theoret. Comput. Sci.}, 304(1--3):59--86, 2003.

\bibitem{LPW2005}
Guy Louchard, Helmut Prodinger, and Mark~Daniel Ward.
\newblock The number of distinct values of some multiplicity in sequences of
  geometrically distributed random variables.
\newblock In {\em 2005 {I}nternational {C}onference on {A}nalysis of
  {A}lgorithms}, Discrete Math. Theor. Comput. Sci. Proc., AD, pages 231--256.
  2005.

\bibitem{MP2011}
Frosso~S. Makri and Zaharias~M. Psillakis.
\newblock On success runs of a fixed length in {B}ernoulli sequences: exact and
  asymptotic results.
\newblock {\em Comput. Math. Appl.}, 61(4):761--772, 2011.

\bibitem{Park2023}
Jinyoung Park.
\newblock Threshold phenomena for random discrete structures.
\newblock {\em Notices Amer. Math. Soc.}, 70(10):1615--1625, 2023.

\bibitem{Prodinger2001}
Helmut Prodinger.
\newblock Combinatorics of geometrically distributed random variables:
  inversions and a parameter of {K}nuth.
\newblock {\em Ann. Comb.}, 5(2):241--250, 2001.

\bibitem{Romik2015}
Dan Romik.
\newblock {\em The surprising mathematics of longest increasing subsequences}.
\newblock Cambridge University Press, 2015.

\bibitem{Wilf2011}
Herbert~S. Wilf.
\newblock The distribution of run lengths in integer compositions.
\newblock {\em Electron. J. Combin.}, 18(2):~Paper 23, 2011.

\end{thebibliography}

\end{document}